\documentclass{amsart}

\usepackage{amssymb}

\usepackage{amsmath}
\usepackage[title]{appendix}
\usepackage{graphicx, enumerate, url}
\usepackage{amssymb,amsmath,stmaryrd}
\usepackage{amsthm,mathtools}
\usepackage{booktabs}
\usepackage{tabularx}
\usepackage{tikz}
\usepackage{pgfplots}
\usetikzlibrary{intersections,pgfplots.fillbetween,arrows.meta,bending}

\usepackage{mathptmx}      

\usepackage{algorithm}
\usepackage{algpseudocode}
\usepackage{diagbox}
\usepackage{ifxetex,ifluatex}
\newif\ifxetexorluatex
\ifxetex
  \xetexorluatextrue
\else
  \ifluatex
    \xetexorluatextrue
  \else
    \xetexorluatexfalse
  \fi
\fi

\ifxetexorluatex
  \usepackage{fontspec}
\else
  \usepackage[T1]{fontenc}
  \usepackage[utf8]{inputenc}
  \usepackage{lmodern}
  \fi
\usepackage{float}

\usepackage{amssymb, stmaryrd}
\usepackage{amsfonts}
\usepackage{mathrsfs}

\usepackage{hyperref}
\usepackage{cleveref}

\newtheorem{proposition}{Proposition}
\newtheorem{lemma}{Lemma}
\newtheorem{theorem}{Theorem}
\newtheorem{corollary}{Corollary}

\theoremstyle{definition}
\newtheorem{example}{Example}

\numberwithin{equation}{section}
\def\VR{\kern-\arraycolsep\strut\vrule &\kern-\arraycolsep}
\def\vr{\kern-\arraycolsep & \kern-\arraycolsep}

\newcommand{\sfT}{\mathsf{T}}

\newcommand{\oo}{\mathfrak{o}}
\newcommand{\piperp}{\pi_{_\perp}}

\newcommand{\fR}{\mathfrak{r}}

\newcommand{\tGamma}{\tilde{\Gamma}}
\newcommand{\ttA}{\mathtt{A}}
\newcommand{\ttH}{\mathtt{H}}
\newcommand{\ttM}{\mathtt{M}}
\newcommand{\ttL}{\mathtt{L}}
\newcommand{\ttX}{\mathtt{X}}
\newcommand{\ttY}{\mathtt{Y}}
\newcommand{\ttZ}{\mathtt{Z}}
\newcommand{\fq}{\mathfrak{q}}

\newcommand{\cE}{\mathcal{E}}

\newcommand{\Two}{\mathbb{II}}

\newcommand{\R}{\mathbb{R}}

\newcommand{\E}{\mathbb{E}}

\newcommand{\sigmam}{\mathring{\sigma}}
\newcommand{\fF}{\mathfrak{F}}

\newcommand{\uV}{\underline{V}}
\newcommand{\oV}{\overline{V}}

\newcommand{\rB}{\mathrm{B}}
\newcommand{\rG}{\mathrm{G}}

\newcommand{\rD}{\mathrm{D}}
\newcommand{\rC}{\mathrm{C}}

\newcommand{\rH}{\mathrm{H}}
\newcommand{\rS}{\mathrm{S}}

\newcommand{\fL}{\mathfrak{L}}
\newcommand{\ehess}{\mathsf{ehess}}
\newcommand{\egrad}{\mathsf{egrad}}
\newcommand{\rgrad}{\mathsf{rgrad}}

\newcommand{\rhess}{\mathsf{rhess}}

\newcommand{\sym}{\mathrm{sym}}

\newcommand{\asym}{\mathrm{skew}}

\newcommand{\cZ}{\mathcal{Z}}
\newcommand{\cV}{\mathcal{V}}
\newcommand{\cW}{\mathcal{W}}

\newcommand{\cB}{\mathcal{B}}

\newcommand{\Herm}[1]{\mathrm{Sym}_{#1}}

\newcommand{\St}[2]{\mathrm{St}_{#1, #2}}
\newcommand{\Sph}{\mathrm{S}}
\newcommand{\Gr}[2]{\mathrm{Gr}_{#1, #2}}

\newcommand{\Sd}[1]{\mathrm{S}^{+}_{#1}}

\newcommand{\cH}{\mathcal{H}}

\newcommand{\cI}{\mathcal{I}}
\newcommand{\bcI}{\bar{\mathcal{I}}}
\newcommand{\pg}{p_{\fG}}
\newcommand{\fGperp}{\fG_{_{\perp}}}

\newcommand{\fG}{\mathfrak{g}}

\newcommand{\cF}{\mathcal{F}}

\newcommand{\cM}{\mathcal{M}}

\newcommand{\cN}{\mathcal{N}}

\newcommand{\cU}{\mathcal{U}}

\newcommand{\sfg}{\mathsf{g}}

\newcommand{\T}{\mathbb{T}}

\DeclareMathOperator{\Lin}{Lin}

\DeclareMathOperator{\ad}{ad}
\DeclareMathOperator{\diag}{diag}
\DeclareMathOperator{\Tr}{Tr}

\DeclareMathOperator{\GL}{GL}
\DeclareMathOperator{\SL}{SL}
\DeclareMathOperator{\Aff}{Aff^+}

\DeclareMathOperator{\SO}{SO}
\DeclareMathOperator{\SE}{SE}

\DeclareMathOperator{\dI}{I}

\begin{document}
\title[SDE on manifolds]{Second-order differential operators, stochastic differential equations and Brownian motions on embedded manifolds}




\author{Du Nguyen}
\email{nguyendu@post.harvard.edu}
\address{Independent, Darien, CT 06820, USA}

\author{ Stefan Sommer}
\email{sommer@di.ku.dk}

\address{University of Copenhagen, Department of Computer Science,
         Universitetsparken 5, Copenhagen 2100,  Denmark}

\begin{abstract}
{We specify the conditions when a manifold $\cM$ embedded in an inner product space $\cE$ is an invariant manifold of a stochastic differential equation (SDE) on $\cE$, linking it with the notion of second-order differential operators on $\cM$. When $\cM$ is given a Riemannian metric, we derive a simple formula for the Laplace-Beltrami operator in terms of the gradient and Hessian on $\cE$ and construct the Riemannian Brownian motions on $\cM$ as solutions of conservative Stratonovich and It{\=o} SDEs on $\cE$. We derive explicitly the SDE for Brownian motions on several important manifolds in applications, including left-invariant matrix Lie groups using embedded coordinates. Numerically, we propose three simulation schemes to solve SDEs on manifolds. In addition to the stochastic projection method, to simulate Riemannian Brownian motions, we construct a {\it second-order tangent retraction} of the Levi-Civita connection using a given $\cE$-tubular retraction. We also propose the {\it retractive Euler-Maruyama method} to solve a SDE,  taking into account the second-order term of a tangent retraction. We provide software to implement the methods in the paper, including Brownian motions of the manifolds discussed. We verify numerically that on  several compact Riemannian manifolds, the long-term limit of Brownian simulation converges to the uniform distributions, suggesting a method to {\it sample Riemannian uniform distributions}.}

\end{abstract}


\maketitle



\section{Introduction}
Diffusion processes play an important role in both theoretical and applied sciences. In machine learning and data science, {\it diffusion maps} and {\it Laplacian-based} methods \cite{CoifmanLafon,BELKINNiy} are important tools for clustering and denoising. Recently, {\it diffusion models} \cite{NEURIPS2020_4c5bcfec} in generative AI could often be considered a discretized version of a stochastic differential equation (SDE). Beyond the Euclidean space, processes on manifolds also have applications in biology, physics \cite{Carlsson_2010}, and data science.  While the data are typically realized as vectors in an Euclidean space, they often need to satisfy certain constraints and symmetries \cite{GDL}. Diffusion processes on manifolds could be used to model constraints and symmetries of the data both directly \cite{GDL,RiemannScore} or in {\it latent spaces} \cite{MENKOVSKI}. {\it Diffusion mean} \cite{SomSvane} is another important application in computer vision.

The theory of stochastic differential equations, in particular, the theoretical construction of Brownian motion on manifolds has been known for a long time \cite{Hsu}. However, there is still a need for an efficient construction of both the equations and the simulations on specific manifolds arising in applications. In the present work, we make several contributions, specifically taking advantage of the embedded structure of the manifold as outlined below. In addition to simulating SDEs, our approach offers a sampling method on compact Riemannian manifolds, which is an important problem in itself.

There is an embedded coordinate description for Riemannian Brownian motions \cite{Hsu}. However, it assumes the metric on the manifold $\cM$ is restricted from the metric on $\cE$. For many practical problems, the embedded metric may not be the {\it natural one}. We address this issue by allowing the use of a metric different from the embedded metric (mathematically, we only require a Whitney embedding, not a Nash embedding). This leads to explicit, simple constructions of Brownian motions on several manifolds used in applied mathematics. In future research, we look to apply our approach to study SDE on manifolds arising in dynamics, robotic, and generative AI.

Two important ingredients of our approach are second-order differential operators and their local description, {\it second-order tangent vectors}. The relationship arises from the fact that the generator of an SDE is a second-order differential operator. 
Second-order tangent vectors, introduced by Schwartz, Meyer, and {\'E}mery \cite{Meyer,Emery2007,Emery1989,Schwartz}, were described more abstractly using exact sequences of bundles. The relationship with acceleration in \cite{Emery2007} allows us to propose a vector-calculus-type description using  global coordinates: they correspond to linear combinations of accelerations, and the constraints imposed on the manifold induce constraints on accelerations, hence on second-order tangent vectors. 
Just as a field of tangent vectors (vector fields) are first-order differential operators, a field of second-order tangent vectors is a second-order differential operator (both with no constant terms), and again, they could be described as global (on $\cE=\R^n$) differential operators satisfying constraints induced by the manifold constraints in \cref{eq:secondtangent}.

For a Riemannian Brownian motion, whose associated generator is the Laplace-Beltrami operator on a Riemannian manifold, we provide a formula for the Laplace-Beltrami operator in the global setup above. The formula involves two ingredients, the metric-compatible projection, and the Levi-Civita connection, and has the same format as the local coordinate formulas for the Laplace-Beltrami formula. Since the work of \cite{Edelman_1999}, it has been appreciated that the global, embedded coordinate formula for the Levi-Civita connection is often convenient in applications. We also describe the Riemannian Brownian motion in both It{\=o} and Stratonovich descriptions, in the case the metric is given by an operator-valued function. We also consider horizontal Brownian motions, with an example of the real Grassmann manifold.

In the above description, the Brownian SDEs on $\cM$ are extended to an SDEs on $\cE$, with $\cM$ as an invariant manifold: if we start at a point on $\cM$, then the process will stay on $\cM$ a.e. In the ordinary differential equation literature, an ODE on an invariant manifold is often solved effectively using {\it projection methods}. The invariant assumption means the exact solution lies on the manifold, and the projected (called retracted here) approximation has a similar order of accuracy as the approximation on $\cE$. This result holds for SDE with some modifications, as has been pointed out by several authors, reviewed in \cref{sec:previous}.

Another way to simulate an SDE is to use {\it tangent retractions}, a mechanism to map a pair of a manifold point and a sufficiently small tangent vector to a new manifold point. In \cite{SHSW}, the authors pointed out that a straightforward simulation of the Riemannian Brownian motion using a tangent retraction would add an {\it extra} second-order term to the generator. Thus, to simulate correctly, they propose using a second-order tangent retraction, (with the extra second-order term vanishes). We make two contributions to this point of view. First, we show starting with any tubular retraction (defined in \cref{sec:tubular}), we can construct a second-order retraction. Alternatively, starting with any SDE of It{\=o}'s form, given a tangent retraction, we can simulate the SDE's solution using the retraction if we modify the drift to account for the extra second-order term. This second result is in line with Meyer and {\'E}mery's notion of geodesic approximation of SDEs, and also with \cite{AmstrongBrigoIntrinsic}'s ideas of representing SDE by jets. The two results give us new alternatives to simulate SDEs with different choices of retractions. We offer an efficient numerical library \cite{LaplaceNumerical} implementing the methods described here.

Closed-form Levi-Civita connections and metric-compatible projections are available for several homogeneous manifolds used in applied mathematics \cite{Edelman_1999,AMS_book,NguyenOperator}.  We express the corresponding Laplace-Beltrami operator and Brownian motions in simple embedded forms. In particular, we treat all left-invariant matrix Lie groups in the same framework, and derive easy-to-evaluate diffusion drifts. We believe the method described here, with the Levi-Civita connection described in \cite{NguyenOperator} will allow us to simulate complex systems in robotics or molecular dynamics. Just as the ability to simulate effectively on a flat space has many real-world applications, we believe the ability to simulate effectively on manifolds will be important in applied problems. As an application, the discussed methods could be used to sample {\it uniform distributions} on compact Riemannian manifolds.

\subsection{Notation}In the following, we use $\rD$ to denote the Euclidean ($\R^n$) directional derivative, $\nabla$ to denote a covariant derivative. We use the Einstein's summation convention extensively, dropping the summation symbols when there is a pair of matching indices. To avoid confusion, we will make the best effort to indicate when it is used.

For two vector spaces $\cE$ and $\cF$, we denote by $\Lin(\cE,\cF)$ the space of linear maps from $\cE$ to $\cF$. We denote the Kronecker delta by $\delta^i_j$. Denote by $\Herm{n}$ or $\Herm{\cE}$
the space of symmetric matrices in $\R^{n\times n}$, or symmetric operators on an inner product space $\cE$. For an operator $P$ on a vector space $\cE$, we write $\omega_P$ for $P(\omega)$ for $\omega\in\cE$ to save space. Thus, $\omega_{\sym} = \frac{1}{2}(\omega + \omega^{\sfT})$,
$\omega_{\asym} = \frac{1}{2}(\omega - \omega^{\sfT})$ for a square matrix $\omega$.

\subsection{Outline}We review previous works in \cref{sec:previous} and related linear algebra and differential geometric background in \cref{sec:background}. In \cref{sec:SOO}, we describe second-order tangent vectors and second-order differential operators in terms of constraints and demonstrate the relationship with invariant manifolds.
We prove our main results on the Laplace-Beltrami operator and Riemannian Brownian motions in \cref{sec:laplaceBeltrami}. In \cref{sec:examples}, we go over several examples of matrix manifolds, in particular, we derive Brownian motion equations on matrix Lie groups. We consider methods to simulate the equations for  Riemannian Brownian motions in \cref{sec:numeric}, including the stochastic projection method and a second-order tangent retraction method. We perform several numerical tests, comparing expected values computed from several simulation methods for the Riemannian Brownian motion equations for different manifolds. We also compare results from Brownian motion simulations with those from integrating the heat kernel for the sphere, and also comparing long-time simulation of Brownian motion with uniform sampling on the sphere, $\SO(n)$, the Stiefel, and the Grassmann manifolds. We end the paper with a brief discussion of future directions.

\section{Previous works}\label{sec:previous} In recent years, stochastic differential equations and Riemannian Brownian motion have found applications in data science. In clustering problems, the concept of diffusion means \cite{eltznerDiffusionMeansGeometric2023} is an important candidate for mean values on Riemannian manifolds. The guided Riemannian Brownian process \cite{JenJoshSommer} allows us to interpolate data points and to learn underlying metrics. Much of this is done in local coordinates. As mentioned, we focus on global coordinates in this paper.

The link between stochastic processes and second-order differential operators has been understood from the previous works of Schwartz, Meyer and {\'E}mery \cite{Meyer,Emery2007,Emery1989,Schwartz}. The link allows us to formulate the {\it stochastic projection method} in \cref{sec:projgeo} using {\it tubular} retractions (essentially a {\it feasibility enhancement} method, defined in \cref{sec:numeric}) following the approaches of \cite{ZhouZhangHongSong,SoloStiefel,SoloWang}.

As mentioned in the introduction, in \cite[theorem 2.2]{SHSW}, the authors show that a straightforward use of a {\it tangent retraction} to simulate a Brownian motion adds an unintended {\it extra } term to the generator of a Riemannian Brownian process (and SDEs in general). Thus, to avoid this extra term, a second-order retraction should be used to simulate the original equation. From \cite{AbM}, the nearest point retraction is a second-order retraction for the embedded metric. Our contribution to this approach is to show this is a special case of {\it our construction of second-order retractions} using 
almost any {\it tubular retraction}. Our method works for both embedded and nonconstant metrics and could be considered a retractive version of the retraction suggested in \cite[section 2.1]{SHSW} for local coordinates.

On the other hand, Armstrong, Brigo, and collaborators analyzed the relationship between second-order differential operators and SDEs in a series of papers. Among them, \cite{AmstrongBrigoIntrinsic} studied It{\=o} equations as jets. They show \cite[theorem 2.4]{AmstrongBrigoIntrinsic} that a 2-jet scheme (a {\it field of curves}) converges to the solution of a SDE. Our {\it retractive Euler-Maruyama's method }\cref{sec:retractivesim} complements this scheme, the iteration in \cref{eq:retracEM} could be considered as a construction of a jet using a given retraction to solve a given equation. The geodesic approximation discussed in \cite{Emery2007} could also be considered related to this result.  For Lie groups, the exponential map is a retraction, thus, could be used with this method if efficient. This result "explains" why the drift does not appear in the geodesic random walk approximation \cite{Gangolli,Jorgensen,SHSW}: the modified drift of \cref{eq:retracEM} is zero if we use a geodesic or its second-order approximation.

In \cite{AMS_book,NguyenRayleigh}, a second-order tangent retraction gives rise to a connection associated with a second-order expansion. If we simulate using a retraction, \cref{theo:REM} shows we need to adjust the drift based on this second-order term, this is consistent  with the works of previous authors that the simulation of an It{\=o} equation on a manifold is dependent on the choice of a connection.

In \cite{ChengZhangSra}, the authors obtained intrinsic estimates for the geometric Euler-Maruyama's method (using the exponent map as the retraction), and took a further step to analyze non-Gaussian noise on manifolds. They obtain a strong order of accuracy of $\frac{3}{2}$ in Riemannian distance, a better estimate than our results for retractions. It will be interesting to extend their intrinsic estimates to retractions.

For a Stiefel manifold $\cM$, \cite{SoloStiefel} derived conditions for a stochastic differential equation on $\cE$ to have $\cM$ as an invariant manifold. We establish here the link between the second-order differential operators on $\cM$ with the invariant manifold condition for $\cM$, extending that result. There are several other pioneering works by the authors of that paper considering the case of the sphere, 
$\SO(N),\SE(N)$ \cite{PiggottSolo}, or alternative iterations on the Stiefel manifold.

The local form of the Laplace-Beltrami operator is known classically \cite[eq 3.3.11]{Hsu}. We show here that the global coordinate formula takes a similar form. For the sphere, this formula is known \cite{Stroock}. A formula for the global Laplace-Beltrami operator appeared in \cite{LaplaceEmb}, where the authors derived the Laplace-Beltrami operator for $\SO(N)$. That formula, derived using the Lagrange multiplier method, is equivalent to ours. We aim to make the relationship with the projection and connection more explicit and to make the relationship with the trace apparent. To our knowledge, the formulas for the Grassmann, Stiefel, and Symmetric Positive Definite manifolds are new. While the Riemannian Brownian process for a Lie group is known \cite{liao2004levy}, we derive new explicit expressions for the It{\=o} and Stratonovich drifts.

Our Stratonovich equation for Riemannian Brownian motion is an extension of \cite[equation 3.2.6]{Hsu}, which requires a Nash embedding. We only require a Whitney embedding, which extends the applicability to larger classes of manifolds used in applications. For example, most left-invariant metrics on a matrix Lie group or its quotient are non-constant (see \cref{eq:gmetric}). The affine-invariant metric on the manifold of positive-definite matrices is not constant, and for $\alpha_0\neq \alpha_1$, the metric in \cref{subsec:stiefel} for a Stiefel manifold is also not an embedded metric. As seen in these examples, the advantage of our approach is we can change the metric without changing the embedding.

The estimate of the rate of convergence of Riemannian Brownian motion to the Riemannian uniform distribution is in \cite{Saloff1994}. As mentioned there, the convergence itself was already well-known.
\section{Background}\label{sec:background}
We will assume familiarity with stochastic processes, ODE and SDE. We review here the main concepts of differential geometry and the linear algebra required. For further reviews of tensor notations, geometry, and Lie groups, the readers can follow textbooks such as \cite{LeeSmooth,ONeil1983,Gallier} and the article \cite{Edelman_1999}. The last article is foundational to the approach of using embedded coordinates for embedded manifolds that we use here. 

\subsection{Some linear algebra: trace, trace sums of bilinear operators and projections}
In Riemannian geometry, the trace is often expressed in terms of an orthonormal basis. However, it is often convenient to use dual bases. The trace of an operator $A$ in a vector space $\cE$ could be computed on any basis $\cB = \{v_i\}_{i=1}^n$ as $\sum_{i=1}^n v^i(Av_i)$, where $v^i$ is the functional sending $w\in\cE$ to $c_i$ if $w = \sum_i c_iv_i$, and the trace is independent of the choice of $\cB$. If $\cE$ is equipped with a nondegenerate $\R$-valued bilinear form $\rB$, then $\{v^1,\ldots,v^n\}$ could be identified with a basis of $\cE$, 
(via the musical isomorphism $\sharp$ \cite[Chapter 11]{LeeSmooth}) such that $\rB(v^i, v_j) = \delta^i_j$, called the dual basis of $\cE$, and the trace is $\sum_{i=1}^n \rB(v^i, Av_i)$. An orthonormal basis is a self-dual basis (with $v_i = v^i$), thus, results for dual bases automatically apply to orthonormal basis.

For two finite-dimensional inner product spaces $\cV$ and $\cW$, their dual spaces $\cV^*$ and $\cW^*$ are identified with $\cV$ and $\cW$ by the inner product, hence if $A$ is an operator between $\cV$ and $\cW$ then its dual $A^*$ is identified with an operator from $\cW$ to $\cV$. We will denote by $A^{\sfT}$ this adjoint map. If we represent $A$ by the matrix $\bar{A}$ by choosing orthonormal bases on $\cV$ and $\cW$ then $A^{\sfT}$ is represented by the transpose matrix $\bar{A}^{\sfT}$.

Let $\cE$ and $\cF$ be two inner product spaces. A $\cF$-valued bilinear form $\Gamma$ on $\cE$ is a map from $\cE\times \cE$ to $\cF$ that is linear in each component. If $\cE$ is identified with $\R^n$, we can represent $\Gamma$ by a collection of matrices (a tensor) as follows. Let $\{f_1,\cdots, f_d\}$ be a basis of $\cF$. Then we can write
$\Gamma(\omega_1, \omega_2)=\sum_{l=1}^df_l\Gamma^l(\omega_1, \omega_2)
$
where each $\Gamma^l(\cdot, \cdot)$
is an $\R$-valued bilinear form, which could be given in matrix form $\Gamma^l(\omega_1, \omega_2)=\omega_1^{\sfT}J_l\omega_2$ for uniquely defined matrices $J_l\in\R^{n\times n}$, $l=1,\cdots, d$. Thus, we have
\begin{equation}\Gamma(\omega_1, \omega_2) = \sum_l f_l\omega_1^{\sfT}J_l\omega_2.\label{eq:gammacomp}
\end{equation}
In particular, $\sum_{i=1}^n\Gamma(v^i, v_i) = \sum_l f_l\Tr(J_l)$ for dual bases $\{v_i\}_{i=1}^n, \{v^i\}_{i=1}^n$ of $\cE$. The following lemmas establish trace properties of this kind of sum. 
\begin{lemma}Let $\{v_i\}_{i=1}^{\dim_{\cV}}$ and $\{w_j\}_{j=1}^{\dim_{\cW}}$ be orthonormal bases of two inner product spaces $\cV, \cW$. Let $\cE$ and $\cF$ be two other inner product spaces. Let $A, B, C$ be three linear operators, where $A$ maps $\cV$ to $\cW$ with adjoint $A^{\sfT}$ mapping $\cW$ to $\cV$, $B$ maps $\cV$ to $\cE$ and $C$ maps $\cW$ to $\cE$. Let $\Gamma$ be a $\cF$-valued bilinear form on $\cE$ as in \cref{eq:gammacomp}. Then $BA^{\sfT}$ maps $\cW$ to $\cE$ and
  \begin{equation}\sum_i\Gamma(CAv_i, Bv_i) = \sum_j \Gamma (Cw_j, BA^{\sfT}w_j).\label{eq:traceManipulation}
\end{equation}      
\end{lemma}
That means we can transform a sum in $\{v_i\}$'s to a sum in $\{w_j\}$'s, by taking transpose and move $A^{\sfT}$ in front of $w_j$. This follows from
$$\begin{gathered}
  \sum_i\Gamma(CAv_i, Bv_i) =\sum_{l} f_l\sum_i  v_i^{\sfT}A^{\sfT}C^{\sfT}J_lBv_i=\sum_l f_l\Tr_{\cV} A^{\sfT}C^{\sfT}J_lB\\
  =\sum_l f_l\Tr_{\cW} C^{\sfT}J_lBA^{\sfT}=\sum_l f_l\sum_jw_j^{\sfT} C^{\sfT}J_lBA^{\sfT}w_j=\sum_j\Gamma(Cw_j, BA^{\sfT}w_j).
\end{gathered} $$ 
\begin{lemma}Let $\cE$ be an inner product space, identified with $\R^n$ with the inner product $(v, w)\mapsto v^{\sfT}w$. Let $\sfg$ be a nondegenerate symmetric matrix defining another pairing $(v, w)\mapsto v^{\sfT}\sfg w$. Let $\{e_i\}_{i=1}^n$ be the standard basis of $\cE=\R^n$. Let $\cV$ be a subspace of $\cE$ of dimension $d$. Let $\{v_i\}_{i=1}^{d}\subset\cV, \{v^i\}_{i=1}^{d}\subset\cV$ be dual bases of $\cV$ with respect to the pairing $\sfg$, thus, $(v^i)^{\sfT}\sfg v_j=\delta^i_j$. Let $\Pi$ be the $\sfg$-compatible projection from $\cE$ onto $\cV$, that means $\Pi$ is an operator on $\cE$ such that $\Pi v = v$ for $v\in\cV$ and $(\Pi e)^{\sfT}\sfg v = e^{\sfT}\sfg v$ for $e\in\cE, v\in\cV$. Then $\Pi$ exists and is unique. If $\rB$ is a bilinear map from $\cE$ to another vector space $\cF$, then we have
  \begin{equation}\sum_{j=1}^d\rB(v_j, v^j)=\sum_{j=1}^d\rB(v^j, v_j) = \sum_{i=1}^n\rB(e_i, \Pi\sfg^{-1}e_i).\label{eq:Bsum}
  \end{equation}
  In particular, the sum is independent of the choice of the dual bases. The lemma holds if we replace $e_i$ with any orthonormal basis of $\cE$ under the standard pairing ($e_i^{\sfT}e_j=\delta^i_j$).  \label{lem:projection}
\end{lemma}
\begin{proof}If $\Pi v = \sum_i c_i v_i$ then $c_i = (\Pi v)^{\sfT}\sfg v^i = v^{\sfT}\sfg v^i$ as $\Pi$ is metric-compatible. This shows $\Pi$ is unique and is equal to $\sum_i (v^{\sfT}\sfg v^i) v_i$, which also verifies existence.
  Let $\uV$ be the $\R^{n\times d}$ matrix formed by columns of $v_i\in\cV\subset\cE=\R^n$, $i=1\cdots d$, and $\oV$ is formed similarly by columns of $v^i$.
  Then $v_i = \uV\epsilon_i, v^i = \oV\epsilon_i$, where $\epsilon_i, i=1\cdots d$ is the $i$-th column of $\dI_d$. The relation $(v^i)^{\sfT}\sfg v_j = \delta_{ij}$ implies $\epsilon_i^{\sfT}\oV^{\sfT}\sfg\uV\epsilon_j = \delta_{ij}$, or $\oV^{\sfT}\sfg\uV=\dI_d$. Applying \cref{eq:traceManipulation} with the bases $\epsilon_j$'s and $e_i$'s (in that setting $\cV=\R^d,\cW=\cE=\R^n$)
  $$\sum_{j=1}^d\rB(v_j, v^j) = \sum_{j=1}^d\rB(\uV\epsilon_j, \oV\epsilon^j) =\sum_{j=1}^n\rB(e_j, \oV\uV^{\sfT} e^j).
  $$
  Thus, it remains to show $\oV\uV^{\sfT}=\Pi\sfg^{-1}$, or $\Pi=\oV\uV^{\sfT}\sfg$. This is just the expression $\Pi v= \sum_i (v^{\sfT}\sfg v^i) v_i$ above.
The last statement of the theorem follows from the previous lemma, as if we replace $e_i$ by $Ue_i$ or an orthogonal matrix $U$, then $\rB(Ue_i, \Pi\sfg^{-1} Ue_i) = \rB(e_i, \Pi\sfg^{-1} UU^{\sfT}e_i) = \rB(e_i, \Pi\sfg^{-1}e_i)$.
\end{proof}  
\begin{lemma}\label{lem:expect}Let $\rB$ be a $\cE_L=\R^k$-valued bilinear form on $\cE=\R^n$. Let $\xi = (\xi_1,\cdots,\xi_n)\sim N(0, \dI_{\cE})$ be a multivariate normal random variable. Then
\begin{gather}\E\rB(\xi, \xi ) = \sum_{i=1}^n\rB_{ii}= \sum_{i=1}^n\rB(e_i, e_i),\label{lem:expect1}\\
\E\lvert v^{\sfT} \xi \rvert  = (\frac{2}{\pi})^{\frac{1}{2}}\lvert v \rvert\text{ for fixed } v\in\cE
\label{lem:expect2},\\
\E\lvert \rB(\xi, \xi ) - \sum_{i=1}^n\rB(e_i, e_i)\rvert^2 = 2\sum_i|B_{ii}|^2 + \sum_{j<i}|B_{ij}+B_{ji} |^2.\label{lem:expect3}
\end{gather}
\end{lemma}
\begin{proof} Write $\rB(\xi, \xi) = \sum\rB_{ij}\xi_i\xi_j$ for $\rB_{ij}\in\cE_L$ then
$\E\rB(\xi, \xi ) = \sum_{ij}\rB_{ij}\E(\xi_i\xi_j)= \sum_{ij}\rB_{ij}\delta^i_j=\sum_i \rB_{ii} = \sum_i\rB(e_i, e_i)$. For the next identity, if $v\neq 0$ then let $u_1$ be the unit vector in direction $v$, and complete it to an orthonormal basis $\{u_1, \cdots u_n\}$. Change to this basis, we can reduce the problem to the case $n=1$ and $v=|v|u_1$, and the second identity follows from $\E|\xi_1|=(\frac{2}{\pi})^{\frac{1}{2}}$. 

Similar to the scalar case, $\E\lvert \rB(\xi, \xi ) - \E\rB(\xi, \xi)\rvert^2 = \E\lvert \rB(\xi, \xi )|^2 - |\E\rB(\xi, \xi)\rvert^2$,
$$
\E\lvert \rB(\xi, \xi ) - \sum_i\rB(e_i, e_i)\rvert^2 =
\E\lvert \rB(\xi, \xi )\rvert^2 - \lvert \sum_i\rB(e_i, e_i)\rvert^2,$$
$$\begin{gathered}\E\lvert \rB(\xi, \xi )\rvert^2 = \sum_i\lvert B_{ii}\rvert^2 \E\xi_i^4 +\sum_{i\neq j}\rB_{ii}^{\sfT}B_{jj}\E\xi_i^2\xi_j^2
+\sum_{i\neq j}(\rB_{ij}^{\sfT}\rB_{ij} +  \rB_{ij}^{\sfT}\rB_{ji})\E\xi_i^2\xi_j^2\\
=3\lvert B_{ii}\rvert^2 + 2\sum_{i < j}\rB_{ii}^{\sfT}B_{jj} + \sum_{i < j}|\rB_{ij}+\rB_{ji}|^2,
\end{gathered}
$$
where we note $\E\xi_i = \E\xi_i^3=0, \E\xi_i^4=3$ for each $i$, $\E\xi_i^2\xi_j^2 = 1$ for all $i, j$ and 
$$\sum_{i\neq j}(\rB_{ij}^{\sfT}\rB_{ij} +  \rB_{ij}^{\sfT}\rB_{ji}) = \sum_{i\neq j}\rB_{ij}^{\sfT}(\rB_{ij} +  \rB_{ji}) = \sum_{i<j}(\rB_{ij} +  \rB_{ji})^{\sfT}(\rB_{ij} +  \rB_{ji}).
$$
Since $\lvert \sum_i\rB(e_i, e_i)\rvert^2 = \sum_i |\rB_{ii}|^2 + 2\sum_{i<j}\rB_{ii}^{\sfT}\rB_{jj} $, we get \cref{lem:expect3}.
\end{proof}
\subsection{Riemannian geometry}
We here give a brief review of concepts in differential geometry relevant to this paper.
More details can be found in \cite{ONeil1983,LeeSmooth}, for example.

Manifolds often appear in applications as smooth constrained sets of an inner product (Euclidean) space, or quotients (equivalent class) of such constrained sets under a group action. Abstractly, manifolds are defined by compatible coordinate charts. The Levi-Civita connection, gradient, and Hessian in the local coordinate picture are classically developed. In this section, we also explain how they are represented in global coordinates.

A Riemannian structure on a manifold $\cM$, could be understood as a smooth, symmetric, positive-definite pairing of tangent spaces of $\cM$, thus, it provides a scalar, symmetric bilinear function on $T_x\cM\times T_x\cM$ for all $x\in\cM$, varying smoothly with $x$. When $\cM$ is a submanifold of an inner product space $\cE$, identified with $\R^n$ (for an integer $n>0$) for convenience, there is another pairing using the inner product of $\cE$, $(\xi, \eta)\mapsto \xi^{\sfT}\eta$, if the tangent vectors $\xi, \eta$ are identified with vectors in $\cE$. The Riemannian pairing is then given by $\xi^{\sfT}\sfg(x)\eta$ for an operator $\sfg(x)$ on $T_x\cM\subset\cE$. We can always extend $\sfg(x)$ to a nondegenerate symmetric bilinear paring on $\cE$ \cite{NguyenGeoglobal}, also called $\sfg(x)$, thus, we have a map $\sfg:\cM\mapsto \Herm{\cE}$, restricting to a positive definite bilinear form on the tangent bundle of $\cM$. This approach allows us to use the coordinates of a {\it Whitney embedding} (an embedding in the differential-geometric sense but does not need to preserve metric), which appears often in applications. For $x\in\cM$, the operator $\sfg(x)$ defines a pairing on $\cE$, denoted by $\langle,\rangle_{\sfg}$ or $\langle,\rangle_{\sfg,x}$, which is the Riemannian pairing when restricted to $T_x\cM$.

The first important ingredient \cite{Edelman_1999} in the embedded picture is the metric-compatible projection as in \cref{lem:projection}. For each $x\in\cM$, denoted by $\Pi(x)$ the metric-compatible projection associated with $\sfg(x)$. Thus, we have a $\Lin(\cE,\cE)$-valued function $\Pi:x\mapsto \Pi(x)$, with each $\Pi(x)$ being idempotent and metric-compatible. For a smooth embedding with smooth $\sfg$, $\Pi(x)$ is smooth \cite{NguyenGeoglobal}. In embedded coordinates, $\Pi(x)$ is a matrix in $\R^{n\times n}$, $\Pi\sfg^{-1}$ is symmetric, and we can write $\Pi =\dI_{\cE}- \sfg^{-1}(\rC')^{\sfT}(\rC'\sfg^{-1}(\rC')^{\sfT})^{-1}\rC'$ if $\cM$ is given by $\rC(x) = 0$ (locally) near $x$.

Given a function $f$ on $\cM$ and a tangent vector $\xi$ at $x\in\cM$, the directional derivative $\rD_{\xi}f$ could be evaluated in embedded coordinates, by extending $f$ to a function $\bar{f}$ (but we will often write as $f$) on $\cE$ and consider $\xi$ as a vector in $\cE$, $\rD_{\xi}\bar{f}$ evaluated in the later picture is the same as $\rD_{\xi}f$ in the local coordinate. The Riemannian gradient at $x$, defined as the tangent vector $\rgrad_f(x)$ such that $(\rD_{\xi}f)(x) = \langle \xi, \rgrad_f(x)\rangle_{\sfg}$ for all $\xi\in T_x\cM$ is given in local coordinates as $(\sfg^{-1}[\frac{\partial f}{\partial z_1},\cdots, \frac{\partial f}{\partial z_d}]^{\sfT})_{|x}$, $d=\dim\cM$ and $(z_1,\cdots, z_d)$ are local coordinate functions. In global ($\cE=\R^n$)-coordinates, if $\egrad_f$ is the gradient of the extended $f$ in the standard coordinates of $\R^n$ then \cite{Edelman_1999}
\begin{equation}\rgrad_f(x) =(\Pi\sfg^{-1}\egrad_f)_{|x}.\end{equation}

\subsubsection{The Levi-Civita connection}\label{subsec:Levi}
The Levi-Civita connection $\nabla$ in local coordinates is given by coefficients $\Gamma_{jk}^i$ such that for a local coordinate system $(z_1\cdots z_d)$, set $\partial_i = \frac{\partial}{\partial z_i}$ then $\nabla_{\partial_j}\partial_k = \Gamma_{jk}^i\partial_i$, with the first requirement is $\nabla_{\partial_j}\partial_k$ is a tangent vector. (Note that we use the Einstein's summation convention in this subsection, so $\Gamma_{jk}^i\partial_i$ means $\sum_i \Gamma_{jk}^i\partial_i$ here). The second condition is torsion-free, $\Gamma_{jk}^i = \Gamma_{kj}^i$. The final condition is metric compatible, $\partial_j\langle\partial_l, \partial_k\rangle_{\sfg} = \langle \Gamma_{jl}^i\partial_i, \partial_k\rangle_{\sfg} + \langle \partial_l, \Gamma_{jk}^i\partial_i\rangle_{\sfg}.$

Globally, we identify vector fields with $\cE$-value functions in global coordinates $(x_1,\cdots, x_n)\in\R^n$, assume $\ttX=X^i\frac{\partial}{\partial x_i}$ and $\ttY= Y^j\partial_j\frac{\partial}{\partial x_i}$ then we have
\begin{equation}
  \nabla_{\ttX}\ttY = \rD_{\ttX}\ttY + \Gamma(\ttX, \ttY)\label{def:Levi}
\end{equation}
where $\Gamma(\ttX, \ttY) = X^jY^k\Gamma_{jk}^i \frac{\partial}{\partial x_i}$ and $\rD_{\ttX}\ttY= (X^i\frac{\partial Y^j}{\partial x_i})\frac{\partial}{\partial x_j}$ is the $\R^n$-derivative if we identify $\ttY$ with a vector-valued function in $\R^n$ then take the derivative in direction $\ttX$. Here, $\Gamma$ is bilinear in $\ttX$ and $\ttY$. If we only require $\nabla_{\ttX}\ttY$ to be a vector field, we have an {\it affine connection}. The additional requirements for the Levi-Civita connection are torsion-free, which means $\Gamma(\ttX, \ttY)=\Gamma(\ttY, \ttX)$, and metric compatibility, which means \begin{equation}\rD_{\ttX}\langle\ttY, \ttY\rangle_{\sfg} =
2\langle\ttY, \nabla_{\ttX}\ttY\rangle_{\sfg}.\label{eq:MetricComp}
\end{equation}
In this operator form, if $\cM\subset\cE$, and we consider $\ttY$ as an $\cE$-valued function from $\cM$, then $\Gamma$ is uniquely defined for vector fields $\ttX$ and $\ttY$  and is given as
\begin{equation}\Gamma(\ttX, \ttY) = -(\rD_{\ttX}\Pi)\ttY + \frac{1}{2}\Pi\sfg^{-1}((\rD_{\ttX}\sfg)\ttY + (\rD_{\ttY}\sfg)\ttX - \chi_{\sfg}(\ttX, \ttY)).\label{eq:GammaLeviCivita}
\end{equation}
where $\chi_{\sfg}$ satisfies $\ttZ^{\sfT}\chi_{\sfg}(\ttX, \ttY) = \ttY^{\sfT}(\rD_{\ttZ}\sfg)\ttX$ for three vector fields $\ttX, \ttY, \ttZ$, see \cite{NguyenOperator}. For $x\in\cM$, $\Gamma$ (dependent on $x$) is uniquely defined on $T_x\cM$ but not on $\cE\times\cE$. We can choose any \emph{bilinear extension} of $\Gamma$ to $\cE\times\cE$, and will call such an extension $\Gamma$ a \emph{Christoffel function} of the Levi-Civita connection $\nabla$. Note, $\Gamma$ depends on the values of $\ttX$ and $\ttY$ and not on their derivatives.

The Euclidean Hessian $\ehess_f=(\frac{\partial^2f}{\partial x_i\partial x_j})_{i,j=1}^n\in\R^{n\times n}$ of a scalar function $f$ on $\cM$ could be considered as a symmetric operator or a bilinear form. On $\cM$, the Riemannian Hessian $\rhess_f$ of $f$ could also be given in two forms (we use tensor conventions and transformations, see \cite[chapter 12]{LeeSmooth}, in particular, bilinear forms are $(0, 2)$ tensors and linear operators on $T\cM$ are $(1, 1)$-tensors). In the scalar form, at $x\in\cM$ and $v_1, v_2\in T_x\cM$ we extend $v_1, v_2$ to vector fields $\ttX, \ttY$, and compute $\rhess^{02}_f(x; v_1, v_2) := \rD_{\ttX}\rD_{\ttY} f - \rD_{\nabla_{\ttX}\ttY} f $, an expression not dependent on the extensions of $\ttX, \ttY$, and $f$ to an open neighborhood of $\cM\subset \cE$. Write $\rD_{\ttY}f = \egrad_f^{\sfT}\ttY$, $\rD_{\nabla_{\ttX}\ttY} f=\egrad_f^{\sfT}\nabla_{\ttX}\ttY$,
$$\begin{gathered}\rD_{\ttX}\rD_{\ttY} f - \rD_{\nabla_{\ttX}\ttY} f   
= \rD_{\ttX} (\egrad_f^{\sfT}\ttY)
 -  \egrad_f^{\sfT}(\rD_{\ttX}\ttY +\Gamma(\ttX, \ttY))\\
=(\rD_{\ttX} \egrad_f)^{\sfT}\ttY + \egrad_f^{\sfT}\rD_{\ttX}\ttY -  \egrad_f^{\sfT}(\rD_{\ttX}\ttY +\Gamma(\ttX, \ttY))\\
= (\rD_{v_1} \egrad_f)^{\sfT}v_2 -\egrad_f^{\sfT}\Gamma(v_1, v_2).
\end{gathered}$$
We get the expression (\cite[eq. (2.57)]{Edelman_1999}, \cite[eq. (3.12)]{NguyenOperator})
\begin{equation}\rhess^{02}_f(x; v_1, v_2) :=\rD_{\ttX}\rD_{\ttY} f - \rD_{\nabla_{\ttX}\ttY} f 
= v_1^{\sfT}(\ehess_f v_2) - \egrad_f^{\sfT}\Gamma(v_1, v_2)\label{eq:rhess02}.
\end{equation}
The Riemannian Hessian vector product \cite{Edelman_1999,AMS_book} is given by 
\begin{equation}\rhess_f^{11}(x)\ttX= \nabla_{\ttX}\rgrad_f.\end{equation}
By the metric compatibility of $\nabla$, it relates to $\rhess_f^{02}$ in the relation below
$$\langle \ttY, \nabla_{\ttX}\rgrad_f\rangle_{\sfg} = \rD_{\ttX}\langle \ttY, \rgrad_f\rangle_{\sfg} - \langle\nabla_{\ttX}\ttY, \rgrad_f \rangle_{\sfg}=\rD_{\ttX}\rD_{\ttY}f-\rD_{\nabla_{\ttX}\ttY}f$$
for a tangent vector field $\ttY$. 
We will suppressed the superscripts ${}^{02}$ and ${}^{11}$ going forward for brevity, to be understood from the context.

The Laplace-Beltrami operator $\Delta_f$ is the trace of the Riemannian Hessian operator $\xi\mapsto \nabla_{\xi}\rgrad_f$ on $T_x\cM$, thus using either expression of $\rhess_f$
\begin{equation}\begin{gathered}\Delta_f = \rhess_f(V^i, V_i)
=\langle V^i, \nabla_{V_i}\rgrad_f \rangle_{\sfg}
  = \rD_{V_i}\rD_{V^i} f -\rD_{\nabla_{V_i}V^i}f\\
  =V_i^{\sfT}(\ehess_f V^i) - \egrad_f^{\sfT}\Gamma(V_i, V^i)
  \label{eq:laplacelocal}
  \end{gathered}
\end{equation}  
(using Einstein's summation convention) for a pair of (locally) dual bases of vector fields $\{V_1,\cdots, V_d\}, \{V^1,\cdots, V^d\}$ near $x\in \cM$, where the last expression is from \cref{eq:rhess02}, and that expression does not involve derivatives of $V_i$'s,  only their values. The {\it local} assumption is because we may not have a global basis of vector fields, (for example, by the hairy ball theorem). As the trace is independent of the choice of basis, we can use a different basis when one local basis eventually becomes degenerate.

Finally, the Riemannian Brownian motion is a diffusion process with generator $\frac{1}{2}\Delta_f$, see \cite{Hsu,Emery2007} for further background.

\subsection{Tubular neighborhoods and retractions}\label{sec:tubular}We will construct numerical algorithms for SDEs on embedded manifolds using {\it retractions}. The term retraction comes in two different, but related contexts, depending on the submanifold setting. We will refer to \cite[chapter 6]{LeeSmooth}, \cite[chapter II.11]{Bredon2013} for backgrounds, and only review the concepts related to our work. Informally, a tubular neighborhood $D_{\cM}\subset\cE$ is an open subset of $\cE$ containing $\cM$ and looks like a tube around $\cM$.
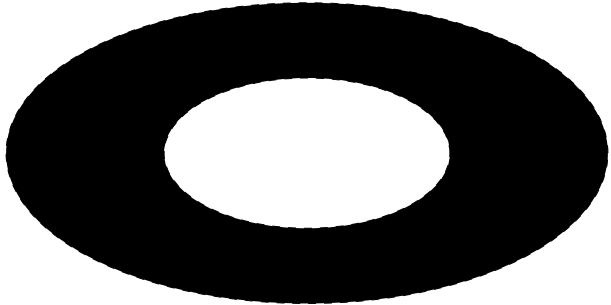
\begin{figure}[htbp!]
\centering
\begin{tikzpicture}[>=latex]
        \draw[dashed,name path=L] (0,0) ellipse (4cm and 2cm)node[left=2.3cm] {$\mathcal{M}$}; 
        \filldraw[black] (3.7, 0) circle(1pt)node[anchor=north] {$q$}; 
        \filldraw[black] (3.0, 0) circle(1pt)node[anchor=east]{$\piperp (q)$};         
        \draw[->,dotted](3.7, 0) -- (3.0, 0);
        \draw[->,dotted](3.7, 0) -- (2.9, .45)node[anchor=east]{$\pi(q)$} ;
        \filldraw[black] (2.9, .45) circle(1pt); 
        
        \draw[->,dotted](3.5, 0) -- (3.0, 0);
        \draw (0,0) ellipse (3cm and 1.5cm);
        \draw[dashed,name path=S] (0,0) ellipse (1.9cm and 1cm); 
        \draw[<->] (0,1.5) -- (0,2) node[right,midway] {$\epsilon$};
        \tikzfillbetween[of=L and S, on layer=ft]{opacity=.2};
    \end{tikzpicture}
\caption{Tubular neighborhood and $\cE$-tubular retraction. The shaded region $D_{\cM}$ is a tubular neighborhood of $\cM$. The point $q\in D_{\cM}$ could be retracted to $\cM$ by the nearest point retraction $\piperp$. It may be more efficient to use another retraction $\pi$.}
\end{figure}
It contains points of a distance not exceeding $\epsilon$ from $\cM$, where $\epsilon$ is small enough so that the closest point is unique. The main existence theorem (tubular neighborhood theorem) \cite[theorem II.11.14]{Bredon2013} requires a compactness condition. However, if we allow varying $\epsilon$, this requirement can be lifted. We will use this setting. For a point $x$ in $D_{\cM}$, there is a unique closest point $\piperp(x)$ to $\cM$, and the map $x\mapsto \piperp(x)$ is an example of a {\it retraction}. A retraction \cite[Proposition 6.25]{LeeSmooth}, which we will call a {\it tubular retraction}, or $\cE$-tubular retraction is a $C^3$-map $\pi$ from $D_{\cM}\subset\cE$ to $\cM$ (the $C^3$ condition could be relaxed) satisfying\hfill\break
\indent 1. $\pi(x)\in\cM$ for $x\in D_{\cM}$,\hfill\break
\indent 2. $\pi(x) = x$ for $x\in\cM$.\hfill\break
Thus, the closest point retraction $\piperp$ is a tubular retraction. We use this general concept to allow other retractions that are easier to compute. We call a tubular retraction $\pi$ an {\it approximated nearest point retraction} (ANP-retraction) if additionally, we have for $x$ sufficiently close to $\cM$\hfill\break
\indent 3.  $|x-\pi(x) |\leq C|x-\piperp(x)|$ for a constant $C$.

From \cite[theorem II.11.14]{Bredon2013}, tubular neighborhoods\slash retractions could be defined for inclusions of the form $\cM\subset\cW\subset\hat{\cE}$ for a submanifold $\cW$ of a vector space $\hat{\cE}$. If we take $\cW=T\cM\subset\cE^2=\hat{\cE}$, (since $T\cM$ is not compact we will use the noncompact version of the cited theorem), this gives us the notion of retraction considered in \cite{AdlerShub,AbM,SHSW}, which we will call a {\it tangent retraction}. The reader can consult for example, \cite[Definition 2.1]{SHSW}. Locally, a tangent retraction maps $(x, v)\in T\cM$ to $\fR(x, v)\in\cM$, with $\fR(x, 0) = x$, with the additional condition $\frac{d}{dt}\{t\mapsto \fR(x, tv)\}_{t=0} = v$.

Both tubular and tangent retractions  are used in the numerical methods discussed in \cref{sec:projgeo}. Further related concepts will be explained there.
\begin{figure}
\centering
\begin{tikzpicture}[scale=2]\footnotesize 
\clip (-1.2,-.3) rectangle (4,1.75);
\begin{scope}[rotate=70]
\coordinate (q) at (0,0);
\draw[name path=ellps] (q) ellipse (1 and 1.2);
\path(0,-.8)--(0,.8)node[left=.8em]{$\mathcal{M}$};
\draw[name path=vertical,-{Stealth}] (1,0)node[above right=1pt] {$x$} -- (1,-.8)node[above right=1pt]{$v$};
\draw [dashed,{-Stealth}](1,-.8) --  (.5, -1.05)node[below]{$\mathfrak{r}(x, v)$};
\draw [dashed,{-Stealth}](1,-.8) --  (.79, -.7)node[left=1pt]{$\piperp(x + v)$};
\end{scope}
\end{tikzpicture}
\caption{Tangent retraction. For small $|v|$, $(x, v)$ belongs to a tubular neighborhood of $\cM$ in $T\cM$, and can be retracted using $\piperp(x + v)$. It may be more efficient to use a different retraction $\mathfrak{r}(x, v)$.}
\end{figure}
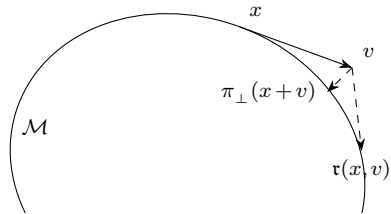
\section{Second-order tangent vectors and second-order differential operators}\label{sec:SOO}
The rank theorem \cite[theorem 4.12]{LeeSmooth} shows that locally, a smooth submanifold $\cM$ of $\cE=\R^n$ (we can replace $\cE$ with another smooth manifold) is given by $n-\dim\cM$ equations \cite[eq. 4.3]{LeeSmooth}. Thus, locally, a submanifold could be identified as a constrained manifold, given by an equation of the form $\rC(x)=0$ from $\cE$ to $\cE_L = \R^{n-\dim\cM}$, where 
the Jacobian $\rC'(x)$ is of rank $n-\dim\cM$. Tangent vectors could be understood as velocity vectors, vectors of the form $\dot{\gamma}(0)$ for a curve $\gamma$ on $\cM$. Alternatively, by differentiating the equation $\rC(\gamma(t)) = 0$, we get $\rC'(\gamma(t))\dot{\gamma}(t) = 0$. From here, the tangent space $T_x\cM$ of all tangent vectors at $x$ is defined by the equation $\rC'(x)v =0$ for $v\in T_x\cM$.

In \cite{Emery2007}, {\'E}mery explained that the second-order tangent space $\T_x\cM$ could be described as the space {\it spanned} by accelerations $\ddot{\gamma}(0)$ of curves starting at $x$, much like tangent vectors could be identified with velocities $\dot{\gamma}(0)$. The difference is that acceleration needs to be linked with velocity, as we can see by differentiating $\rC'(\gamma(t))\dot{\gamma}(t) = 0$
\begin{equation}\rC^{(2)}(\gamma(t); \dot{\gamma}(t), \dot{\gamma}(t)) + \rC'(\gamma(t))\ddot{\gamma}(t)=0.\label{eq:curveSecond}
\end{equation}
where $\rC^{(2)}$ is the second derivative (Hessian) of $\rC$. Thus, the equation for $\ddot{\gamma}(t)$ has a {\it quadratic term} in $\dot{\gamma}(t)$. For curves $\gamma_1,\cdots, \gamma_p$ and coefficients $c_1,\cdots c_p$, $p$ is a nonnegative integer, set $\ttA = \sum_{i=1}^pc_i\ddot{\gamma}_i(0)\in\cE$, and $\ttM := \sum_{i=1}^pc_i\dot{\gamma}_i(0)\dot{\gamma}_i^{\sfT}(0)\in \Herm{n}$ is symmetric, with entries $\ttM^{ij}$. The constraints on $\ttA$ and $\ttM$ are obtained by taking a linear combination of \cref{eq:curveSecond}
\begin{equation}
  \begin{gathered}
\sum_{ij}\rC^{(2)}_{ij}(x)\ttM^{ij} + \rC'(x)\ttA =0,\\
    \rC'(x)\ttM = 0.
  \end{gathered}  \label{eq:secondtangent}
\end{equation}
We call the pair $\ttL = (\ttA, \ttM)\in\R^n\times\Herm{n}$ a second-order tangent vector at $x$ when the constraints above are satisfied. The space of all second-order tangent vectors is identified with the subspace $\T_x\cM$ of dimension $d+\frac{d(d+1)}{2}\subset \R^n\times\Herm{n}$, defined by \cref{eq:secondtangent}. Express $\rC$ as a vector $[\rC^1,\cdots \rC^{n-d}]$, the Hessian of each $\rC^i$ is given by a matrix, and the first equation becomes
\begin{equation}\Tr((\rC^i)^{(2)}(x)\ttM) + (\rC^i)'(x)\ttA=0\quad\text{ for }i=1\cdots n-d.\label{eq:secondtangentdetail}
\end{equation}
If $\ttM = 0$, then $\rC'(x)\ttA =0$, thus, $T_x\cM\times\{0\} \subset\T_x\cM$, and $\rC'(x)\ttM=0$ shows $\ttM$ maps to $T_x\cM$, or it defines a bilinear pairing on $T_x\cM$.

If an open subset of $\cM$ is identified with an open subset of $\R^{\dim\cM}$ then this coincides with the approach in \cite{Emery2007,Emery1989}, where $\ttA$ and $\ttM$ are given by a vector $(L^i)_{i=1}^d$ and a symmetric matrix $(L^{ij})_{i,j=1}^d$.

When we take directional derivative in global coordinates, the result does not depend on the extension of $f$ to $\cE$ if the direction is tangent to $\cM$. The situation is similar for second-order tangent vectors.
\begin{lemma}Let $(A, M)$ be a second-order tangent vector at $x\in\cM$. For any $C^2$ function $f$ defined in an $\cE$-neighborhood $\Omega$ of $x$, define
\begin{equation}
  \ttL: f\mapsto (\ttL f)(x) = \egrad_f(x)^{\sfT}A + \Tr(\ehess_f(x)M)).
\end{equation}
Then $\ttL f(x) = 0$ if $f$ restricts to $\Omega\cap \cM$ is zero, thus $\ttL f$ does not depend on values of $f$ outside of $\cM$.
\end{lemma}  
\begin{proof}If $M=0$ then $A$ is a tangent vector and $\ttL f = \egrad_f(x)^{\sfT} A =(\rD_A f)(x)$ clearly satisfies the lemma. Consider the case $M\neq 0$. If $(c,v)$ is an eigenpair of $M$ with $c\neq 0$ then $0=\rC'(x) M v = c\rC'(x) v$, thus, $v$ is a tangent vector. Thus, if $(c_1,v_1), \cdots, (c_p, v_p)$ are nonzero eigenpairs with orthonormal eigenvectors then $M = \sum_{i=1}^p c_i v_iv_i^{\sfT}$ and $v_1,\cdots, v_p$ are tangent vectors. Consider a parametrization of $\cM$ near $x$ with a chart $\phi: U\mapsto \phi(U)\subset \cM$ for an open subset $U\subset \R^d$. Consider $p$ curves $(\alpha_1,\cdots \alpha_p)$ in $U$ such that $\alpha_i(0)=\phi^{-1}(x)$, $\phi'(\phi^{-1}(x))\dot{\alpha}_i(0) = v_i$. Then $\gamma_i(t):=\phi(\alpha_i(t))$ are curves on $\cM$ with $\gamma_i(0) = x$, $\dot{\gamma}_i(0) = v_i$, thus $\ttL_{\circ} = (A_{\circ}, M) := (\sum_i c_i\ddot{\gamma}_i(0), M)$ is a second-order tangent vector, and $\ttL_{\circ} f(x) = (\sum_i c_i\frac{d^2}{dt^2} f(\gamma_i(t)))_{t=0}=0$, while $\rC'(x)(A-A_{\circ})=0$, hence $A-A_{\circ}$ is a tangent vector. Thus, $\ttL f(x) = \ttL_{\circ} f + \rD_{A-A_{\circ}} f = 0$.
\end{proof} 
A {\it section} of $\T\cM$, or a {\it field of second-order tangent vectors} could be considered a function $\ttL$ from $\cM$ to $\R^n\times\Herm{n}$ such that $\ttL(x) = (\ttA(x), \ttM(x))\in \T_x\cM$.
A {\it second-order operator with zero constant term} (abbreviated SOO) is a section of $\T\cM$. The lemma shows SOO operates on functions on $\cM$.

Without constraint, the Laplace operator on $\R^n$ corresponds to the SOO $\ttL=(0_n, \dI_n)\in \T_x\R^n$ for $x\in\R^n$. 

From \cite{Meyer,Emery2007}, generators of solutions of SDE on manifolds are SOO. There is not a unique SDE with a given generator. We call an SDE to be associated with a SOO $L$ if its generator is $\frac{1}{2}L$. In particular, the equation $dX_t = dW_t$ is associated with the Laplacian $(0, \dI_n)$ on $\R^n$.

On $\R^n$, the It{\=o} equation $dX_t = \mu(X_t, t) dt + \sigma dW_t$ is associated with $(2\mu, \sigma\sigma^{\sfT})$. For a vector field $\mu_S$, the Stratonovich equation $dX_t = \mu_S dt+ \sum_{\gamma}\sigma_{\beta}\circ dW_t^{\beta}$ corresponds to the operator $2\mu_Sf + \sum_{\beta}\rD_{\sigma_{\beta}}(\rD_{\sigma_{\beta}}f)$. We have
\begin{equation}(\ttA, \ttM) = (2\mu_S + \sum_{\beta=1}^m \rD_{\sigma_{\beta}}\sigma_{\beta}, \begin{bmatrix}\sigma_1,\cdots, \sigma_m\end{bmatrix} \begin{bmatrix}\sigma_1^{\sfT}\\\vdots\\ \sigma_m^{\sfT}\end{bmatrix}),\end{equation}
 by expanding $\rD_{\sigma_{\beta}}(\sigma_{\beta} f)$, consistent with the relationship $\mu = \mu_S +\frac{1}{2}\sum \rD_{\sigma_{\beta}}\sigma_{\beta}$ between the It{\=o} and Stratonovich drifts. If $\cM\subset \cE$ is a submanifold, with $\mu_S$ and $\sigma_i$'s are vector fields on $\cM$, this is an SOO, as we can verify 
$$(\rC^i)'(2\mu_S + \sum_{\beta=1}^m \rD_{\sigma_{\beta}}\sigma_{\beta}) + \Tr(\rC^i)^{(2)}(\sum_{\beta=1}^m\sigma_{\beta}\sigma_{\beta}^{\sfT} )=0,$$
which follows by differentiating $(\rC^i)'\sigma_{\beta} = 0$ in the tangent direction $\sigma_{\beta}$, and the fact that $\mu_S$ is tangent in the Stratonovich case.

Recall an invariant manifold of a SDE on $\R^n$ is a manifold such that a process starting at a point $X_0$ on $\cM$ will be on $\cM$ a.e. Below, we use the symbol $\sigmam$ instead of just $\sigma$ to signify the fact that the process has an invariant manifold.
\begin{proposition}The It{\=o} process $dX_t =\mu(X_t, t) dt + \sigmam(X_t, t) dW_t$ defined on $\R^n$ near $\cM$, where $\sigmam^{\sfT}\sigmam$ is invertible has $\cM$ as an invariant manifold if and only if $(\ttA, \ttM) = (2\mu, \sigmam\sigmam^{\sfT})$ restricts to a SOO on $\cM$.\label{prop:conservedIto}
\end{proposition}

\begin{proof} Using It{\=o}'s lemma (on $\R^n$), we get for $i=1\cdots n-\dim\cM$
  $$d\rC^i(X_t) = ((\rC^i)'\mu +\frac{1}{2}\Tr(\rC^i)^{(2)}\sigmam\sigmam^{\sfT})dt + (\rC^i)'\sigmam dW_t $$
where we omit the variable names on the right-hand side for brevity. Thus, if $(2\mu, \sigmam\sigmam^{\sfT})$ is an SOO then the first term is zero, while the condition $(\rC^i)'M=(\rC^i)'\sigmam\sigmam^{\sfT} = 0$ implies $(\rC^i)'\sigmam\sigmam^{\sfT} \sigmam =0$, together with $\sigmam^{\sfT}\sigmam$ is invertible means the second term is zero, hence $d\rC^i(X_t)=0$. Since $\rC^i(X_0)=0$, this implies $\rC^i(X_t)$ 
is zero a.e. Conversely, if $d\rC^i=0$ then both the drift and $(\rC^i)'\sigmam=0$, implying \cref{eq:secondtangent}.
\end{proof}  
\section{The Laplace-Beltrami operator as SOO and Riemannian Brownian motions}\label{sec:laplaceBeltrami}
We now give an explicit expression for the Laplace-Beltrami operator in global coordinates in terms of the familiar operators $\Pi, \sfg^{-1}$ and a Christoffel function $\Gamma$. In the next theorem, $\Pi\sigma$ corresponds to $\sigmam$ in \cref{prop:conservedIto}.
\begin{theorem}\label{theo:laplace} Let $\{e_i\}_{i=1}^n$ be the standard orthonormal basis of $\R^n$. For a Whitney-embedded manifold $\cM\subset\cE=\R^n$ with the Riemannian metric given by the operator $\sfg:\cM\mapsto \Herm{\cE}$ having the compatible projection $\Pi$ and a Christoffel function $\Gamma$, the Laplace-Beltrami operator on $\cM$  is given by
  \begin{equation}\Delta_f = \egrad_f^{\sfT}(-\sum_{i=1}^n\Gamma(e_i, \Pi\sfg^{-1}e_i)) +\Tr_{\cE}\Pi\sfg^{-1}\ehess_f\label{eq:embeddedLaplace}
  \end{equation}
  with both sides evaluated at $x\in\cM$. Here, we consider $\Pi\sfg^{-1}\ehess_f$ as an operator on $\cE$ when taking the trace $\Tr_{\cE}\Pi\sfg^{-1}\ehess_f=\sum e_i^{\sfT}\Pi\sfg^{-1}\ehess_fe_i$. In other words, the SOO is given by
  \begin{equation}
    \ttL = (\ttA, \ttM) = (-\sum_{i=1}^n\Gamma(e_i, \Pi\sfg^{-1}e_i), \Pi\sfg^{-1}).
  \end{equation}
 We have $\sum_{i=1}^n\Gamma(e_i, \Pi\sfg^{-1}e_i)) = \sum_{i=1}^n\Gamma(\Pi e_i, \Pi\sfg^{-1}e_i)$.
  
 Assume $\sigma$ is a smooth map defined in a tubular neighborhood $\cU$ of $\cM$, with value $\sigma(x)$ in the space of linear operators on $\cE$, such that $\Pi\sigma(x)\sigma(x)^{\sfT}\sfg\Pi=\Pi$. Let $W_t$ be a Wiener process on $\cE$. The It{\=o} process on $\cE$
  \begin{equation}
    dX_t =  -\frac{1}{2}\sum_{i=1}^n\Gamma(X_t; e_i, (\Pi\sfg^{-1})(X_t)e_i) dt + (\Pi\sigma)(X_t) dW_t
    \label{eq:BrownianIto}
\end{equation}    
has $\cM$ as an invariant manifold, and $X_t$ is the Riemannian Brownian motion on $\cM$ with metric induced by $\sfg$. 

For $\omega\in\cE$, let $\Pi \omega$ denote the vector field $y\mapsto \Pi(y)\omega$ on $\cM$, then the Stratonovich equation below also describes a Riemannian Brownian motion with the metric $\sfg$ on $\cM$. 
  \begin{equation}dX_t = -\frac{1}{2}\sum_{i=1}^n(\nabla_{\Pi\sigma e_i}\Pi\sigma e_i)_{X_t} + (\Pi\sigma)_{X_t}\circ dW_t.
\label{eq:BrownianStratAlt}    
  \end{equation}
Alternatively, assuming the stronger condition $\Pi\sigma\sigma^{\sfT} = \Pi\sfg^{-1}$, then the following  Stratonovich SDE also defines a Brownian motion on $\cM$
  \begin{equation}
    \begin{gathered}  dX_t 
       = -\frac{1}{2}\sum_{i=1}^n\Pi_{_{X_t}}\Gamma_{_{X_t}}(e_i, (\Pi\sfg^{-1})_{_{X_t}}e_i) + \Pi_{_{X_t}}\circ (\sigma_{_{X_t}} dW_t).
      \label{eq:BrownianStrat}
\end{gathered}      
  \end{equation}
\end{theorem}
In particular, if $\sfg$ is positive-definite on $\cE$, we can take $\sigma = \sfg^{-\frac{1}{2}}$, the symmetric positive-definite square root of $\sfg$. 
 Also, for the embedded metric $\sfg = \dI_{\cE}$, and $\sigma =\dI_{\cE}$, \cref{eq:BrownianStratAlt} reduces to \cite[equation 3.2.6]{Hsu}
  \begin{equation}dX_t = \Pi_{X_t}\circ dW_t.
\label{eq:BrownianHsu}    
  \end{equation}  
 
 Let $\tGamma$ be the Levi-Civita connection of an extension of $\sfg$ to $\cU$, we also have $\sum_{i=1}^n\Pi(X_t)\tGamma(e_i, \Pi\sfg^{-1}e_i) = \sum_{i=1}^n\Pi(X_t)\Gamma(e_i, \Pi\sfg^{-1}e_i)$, giving us another way to compute the Stratonovich drift in \cref{eq:BrownianStrat}.
 
 \begin{proof} Using the last expression for $\Delta_f$ in \cref{eq:laplacelocal}, 
    $\Delta_f =\sum_{i=1}^{\dim \cM} \ehess_f^{02}(V_i, V^i) -\egrad_f^{\sfT}(\sum_{i=1}^{\dim\cM}\Gamma( V_i, V^i))$, for dual bases $\{V_i\}$ and $\{V^i\}$ of the tangent space near a point $x\in \cM$. Since $\ehess_f^{02}(., .)$ and $\egrad_f\Gamma(., .)$ are both bilinear,
    we can apply \cref{eq:Bsum}, replacing the sum over the dual bases with the sum over $e_i, \Pi\sfg^{-1}e_i$. \emph{We will use Einstein's convention} for the rest of the proof
    $$\Delta_f =\ehess^{02}_f(e_i, \Pi\sfg^{-1}e_i) - \egrad_f^{\sfT}\Gamma(e_i, \Pi\sfg^{-1}e_i),
    $$
    we get \cref{eq:embeddedLaplace}. Note, this reduces to the local coordinate formula $\Delta_f  =  g^{ij}\partial_i\partial_jf-g^{jk}\Gamma^i_{jk}\partial_if$ on an open subset of $\cE$.
 
 Since $\Pi^2=\Pi$ and $\Pi\sfg^{-1}$ is symmetric, $\sfg^{-1}\Pi^{\sfT} = \Pi\sfg^{-1}$. Using \cref{eq:traceManipulation},
 $$\Gamma(\Pi e_i,\Pi\sfg^{-1}e_i)=\Gamma( e_i,\Pi\sfg^{-1}\Pi^{\sfT}e_i)
 =\Gamma( e_i,\Pi^2\sfg^{-1}e_i)=\Gamma( e_i,\Pi\sfg^{-1}e_i).
 $$
   To show $\ttL$ is an SOO on $\cM$, take $f=\rC^j, j=1\cdots n-\dim\cM$. Since $\rC^j$ is zero on $\cM$, its Riemannian gradient is identically zero, thus, $\Delta_{\rC^j}=\langle V^i, \nabla_{V_i}\rgrad_{\rC^j}\rangle_{\sfg}=0$ in \cref{eq:laplacelocal}. But it is also given by \cref{eq:embeddedLaplace} as 
   $$\Delta_{\rC^j}=(\rC^j)' (-\Gamma(e_i, \Pi\sfg^{-1}e_i)) +\Tr_{\cE}\Pi\sfg^{-1}(\rC^j)^{(2)}
   $$
   which must also be zero. This gives us \cref{eq:secondtangentdetail} in the SOO condition. The other condition $\rC'\Pi\sfg^{-1}=0$ follows as $\Pi$ projects to the tangent space. 
For the It{\=o} equation, the generator of \cref{eq:BrownianIto} for a smooth function $f$ is 
$$\ttL f =-\frac{1}{2}\egrad_f^{\sfT}\Gamma(e_i, \Pi\sfg^{-1}e_i) +\frac{1}{2}\Tr(\Pi\sigma\sigma^{\sfT}\Pi^{\sfT}\ehess_f) .$$
Again $\sfg^{-1}\Pi^{\sfT} = \Pi\sfg^{-1}$ implies $\Pi^{\sfT} = \sfg\Pi\sfg^{-1}$, thus, the last term reduces to
$\frac{1}{2}\Tr(\Pi\sigma\sigma^{\sfT}\sfg\Pi\sfg^{-1}\ehess_f)=\frac{1}{2}\Tr(\Pi\sfg^{-1}\ehess_f).
$
Hence, $\ttL f = \frac{1}{2}\Delta_f$.

If $\{ e_i\}_{i=1}^n$ is the standard basis, then the $i$-th column of $\Pi$ is $\Pi e_{i}$.   As an equation in $\R^n$ and still with Einstein's summation convention, the It{\=o} form of (\ref{eq:BrownianStratAlt}) is exactly \cref{eq:BrownianIto}, since
$$\begin{gathered}-\frac{1}{2}(\nabla_{\Pi\sigma e_i}\Pi\sigma e_i)_{X_t} 
+\frac{1}{2}(\rD_{\Pi\sigma e_i}\Pi\sigma e_i)_{X_t} +
(\Pi\sigma)_{X_t} dW_t\\
=-\frac{1}{2}\Gamma(\Pi\sigma e_i, \Pi\sigma e_i)_{X_t} +
\Pi(X_t)\sigma(X_t) dW_t.
\end{gathered}$$
The It{\=o} form of equation (\ref{eq:BrownianStrat}) is $dX_t = \mu dt + \Pi\sigma dW$, with
$$\begin{gathered}
  \mu dt = -\frac{1}{2}\Pi\Gamma( e_i, \Pi\sfg^{-1}e_i)dt + \frac{1}{2}\rD_{\Pi e_{\alpha}}(\Pi e_{\beta})[(\sigma dW)^{\alpha}, (\sigma dW)^{\beta}]\\
  = -\frac{1}{2}\Pi\Gamma(e_i, \Pi\sfg^{-1}e_i)dt
  + \frac{1}{2}\rD_{(\sigma\sigma^{\sfT})^{\alpha\beta}\Pi e_{\alpha}}(\Pi e_{\beta})dt\\
=   -\frac{1}{2}\Pi\Gamma(\Pi e_i, \Pi\sfg^{-1}e_i)dt
+ \frac{1}{2}\rD_{\Pi\sfg^{-1}e_{\beta}} (\Pi e_{\beta})dt
\end{gathered}
$$
where we use $(\sigma\sigma^{\sfT})^{\alpha\beta}\Pi e_{\alpha} = \Pi\sigma\sigma^{\sfT} e_{\beta}
=\Pi\sfg^{-1}e_{\beta}.$ From here, at $x\in\cM$, since for two tangent vectors $v_1, v_2$, $\nabla_{\Pi v_1}\Pi v_2 \in T_x\cM$, thus $\nabla_{\Pi v_1}\Pi v_2 = \Pi\nabla_{\Pi v_1}\Pi v_2 $
$$\begin{gathered} (\rD_{v_1}\Pi)v_2 + \Gamma(v_1,  v_2) = \Pi(\rD_{v_1}\Pi)v_2 + \Pi\Gamma(v_1,  v_2),\\
\Gamma(v_1,  v_2) = \Pi(\rD_{v_1}\Pi)v_2 - (\rD_{v_1}\Pi)v_2  + \Pi\Gamma(v_1,  v_2).
\end{gathered}$$
But $\rD_{v_1}\Pi v_2 = \rD_{v_1}\Pi^2 v_2 = (\rD_{v_1}\Pi)\Pi v_2 + \Pi\rD_{v_1}(\Pi v_2) =(\rD_{v_1}\Pi) v_2 + \Pi\rD_{v_1}\Pi v_2 $, we have $ \Pi\rD_{v_1}\Pi v_2=0$, hence
$\Gamma(v_1,  v_2) = - (\rD_{v_1}\Pi)v_2  + \Pi\Gamma(v_1,  v_2)$, then
the above reduces to (with the help of \cref{eq:traceManipulation})
$$\mu=-\frac{1}{2}\Gamma(\Pi\sfg^{-1} e_i, \Pi e_i) = -\frac{1}{2}\Gamma(\Pi\sfg^{-1} e_i, e_i) = -\frac{1}{2}\Gamma(e_i, \Pi\sfg^{-1} e_i).$$

Finally, for the embedded metric $\sigma(x) = \dI_{\cE}$, we have $\Gamma(\ttX, \ttY) = -(\rD_{\ttX} \Pi)\ttY$ from \cite{AMT},\cref{eq:GammaLeviCivita}, thus
$$\nabla_{\Pi e_i}\Pi e_i = \rD_{\Pi e_i}\Pi e_i - (\rD_{\Pi e_i} \Pi)\Pi e_i = 0,
$$
as \cref{eq:traceManipulation} gives $(\rD_{\Pi e_i} \Pi)\Pi e_i = (\rD_{\Pi^2 e_i} \Pi) e_i$. This gives us \cref{eq:BrownianHsu}.
\end{proof}  
\begin{example}Consider the unit sphere $\rS^{n-1}$ with the induced constant metric $\sfg = \dI_n$. It is well-known $\Pi(x) = \dI_n - xx^{\sfT}, \Gamma(x; \xi, \eta) = x\xi^{\sfT}\eta$, thus
  $$\sum_i\Gamma( e_i, \Pi \sfg^{-1} e_i) = \sum_i x(e_i^{\sfT}\Pi e_i) = x\Tr\Pi= (n  - \Tr x^{\sfT}x)x =  (n-1)x.$$
  Hence, the Laplacian is $\Delta_f = \Tr(\ehess_f\Pi) - (n-1)\egrad_f^{\sfT}x$, and the Brownian motion could be given as the solution \cite[chapter 3, p.83]{Hsu} of
  \begin{equation}
    dX_t = -\frac{n-1}{2}X_t dt + (\dI_n - X_tX_t^{\sfT})dW_t.
  \end{equation}
  We have $x^{\sfT}\ttM = x^{\sfT}\Pi = 0$ and with $\rC'(x) = 2x^{\sfT}, \rC^{(2)}(x) = 2\dI_n$,   
  $$2x^{\sfT}(-(n-1)x) + \Tr(2\dI_n(\dI_n - xx^{\sfT}))=0,$$ 
confirming this Laplacian is a second-order differential operator on the sphere.
\end{example}
\begin{example} In the upper-half plane model $\rH^{n}$ of hyperbolic geometry with coordinate $(x_1,\cdots, x_n)^{\sfT}, x_n > 0$, with $\sfg=\frac{1}{x_{n}^2}\dI_{n}$, $\Pi=\dI_n$ is the identity map, and $\Gamma(x; \xi, \eta) = -\frac{1}{x_n}(\xi_n\eta + \eta_n\xi -\xi^{\sfT}\eta e_n)$, where $\xi_n, \eta_n$ are the $n$-th coordinate of $\xi,\eta\in \R^n$ and $e_n=(0,\cdots,0,1)^{\sfT}$. Hence, $\sum_i\Gamma(x; e_i, x_n^2e_i) = (n-2)x_ne_n$, which gives $\Delta_f =x_n^2\sum_{i=1}^n \frac{\partial^2}{\partial x_i^2} -(n-2)x_n\frac{\partial}{\partial x_n}$ and the Brownian SDE
  \begin{equation}
    dX_t = -\frac{n-2}{2}X_{n,t}e_n + X_{n,t} \dI_n dW_t.
\end{equation}    
\end{example}
\subsection{Lifting of Brownian motion}\label{sec:hBrown}
Consider an embedded manifold $\cM\subset\cE$, and assume we have a {\it Riemannian submersion } $\fq:\cM\to\cB$ (see \cite[chapter 7]{ONeil1983} for background material). For $x\in\cM$, we have a submanifold $\fq^{-1}(\fq(x))$, whose tangent space is a subspace $\cV_x \subset T_x\cM$, called the vertical bundle. Its orthogonal complement $\cH_x$ is called the horizontal space, the corresponding subbundle is called the horizontal bundle, and the projection from $\cE$ to $\cH_x$ is called the {\it horizontal projection.} The corresponding vector fields are called horizontal. In a Riemannian submersion, the map $\fq'(x)$ is an isometry from the horizontal space to $T_{\fq(x)}\cB$, thus, any vector $w\in T_{\fq(x)}\cB$ corresponds to a unique horizontal vector at $x\in\cM$, called the {\it horizontal lift}. We can lift Riemannian Brownian processes. From the proof of theorem 4.1.10 in \cite{Baudoin}, the lift satisfies
\begin{equation}dX_t = \sum_i H_i\circ dW^i_t -\frac{1}{2}\sum_i(\nabla^{\cH}_{H_i}H_i) dt\label{eq:lifted}
\end{equation}
for basic horizontal orthonormal vector fields $H_i$ (see \cite[Definition 4.1.7]{Baudoin} for definition), where the lifted connection $\nabla^{\cH}$ is defined on the horizontal bundle $\cH$, and for a tangent vector field $\ttX$ and a horizontal vector field $\ttY$ \cite[lemma 7.45]{ONeil1983} $\nabla^{\cH}_{\ttX}\ttY = \rH\nabla_{\ttX}\ttY$, where $\rH$ is the horizontal projection, and $\nabla$ is the Levi-Civita connection on $\cM$. 
We can write
\begin{equation}
\nabla^{\cH}_{\ttX}\ttY = \rD_{\ttX}\ttY + \Gamma^{\cH}(\ttX, \ttY)
\end{equation}
for horizontal vector fields $\ttX, \ttY$, where $\Gamma^{\cH}$ (depending on $x$) is bilinear in $\ttX$ and $\ttY$, valued in $\cE$. The vector field $\nabla^{\cH}_{\ttX}\ttY$ is horizontal.

For a submanifold $\cN\subset\cM$, let $\Pi_{\cN}$ be the projection from $\cE$ to $T\cN$. For
two vector fields $\ttX,\ttY$ on $\cN$, $\Two(\ttX, \ttY) = (\Pi_{\cM} - \Pi_{\cN})\nabla_{\ttX}\ttY$ is called the second fundamental form. It is a tensor, whose trace is called the mean curvature
\begin{equation}
\mathbb{H} = (\Pi_{\cM} - \Pi_{\cN})\sum_{j=1}^{\dim\cN}\nabla_{\ttZ_j}\ttZ_j = \sum_{i=1}^n\Two(e_i, \Pi_{\cN}\sfg^{-1}e_i)
\end{equation}  
for a locally orthonormal frame $\{\ttZ_i\}$ of $\cN$, using \cref{eq:Bsum}, if we extend $\Two$ bilinearly to $\cE$.

For a Riemannian submersion $\fq$, the horizontal Laplacian \cite{Baudoin} is
\begin{equation}\Delta^{\cH}_f = \sum_i\rD_{H_i}\rD_{H_i}f -\sum_i\rD_{\ttH\nabla_{H_i}H_i}f - \sum_j\rD_{\ttH\nabla_{Z_j}Z_j}f
\end{equation}
at $x\in\cM$, for an orthonormal frame $\{H_i\}$ of the horizontal bundle $\cH$, and a vertical frame $\{\cZ_j\}$, which is an orthonormal frame of the submanifold $\cN_x = \fq^{-1}\fq(x)$. Thus, if the mean curvature of $\fq^{-1}\fq(x)$ is zero, that means (with Einstein's convention here and below)
$\ttH\nabla_{Z_i}Z_i=(\Pi_{\cM} - \Pi_{\cN})\nabla_{\ttZ_i}\ttZ_i=0$ for $\cN= \fq^{-1}\fq(x)$, $\Delta^{\cH}$ is equal to the generator of the lifted process \cref{eq:lifted} $\ttL_{\cH} f = \rD_{H_i}(\rD_{H_i} f) - \rD_{\rH\nabla_{H_i}H_i}f$ \cite[chapter 4]{Baudoin}. The horizontal Brownian motion is a process with generator $\frac{1}{2}\Delta^{\cH}$. Following the proof of \cref{theo:laplace}, we have
\begin{theorem}\label{theo:lift}Assume the mean curvature of every vertical submanifold $\fq^{-1}(\fq(x))$ is zero. Then the second-order differential operator corresponding to the lifted Brownian motion \cref{eq:lifted}, on a lifted function $f$ is the same as that of the horizontal Brownian motion
  \begin{equation}
\Delta^{\cH}_f = \ttL^{\cH}f  = -\egrad_f^{\sfT}\sum_{i=1}^{\dim \cE}\Gamma^{\cH}(e_i, \rH\sfg^{-1} e_i) + \Tr(\rH\sfg^{-1}\ehess_f).
  \end{equation}
  Note $\sum_i\Gamma^{\cH}(e_i, \rH\sfg^{-1} e_i)=\sum_i\Gamma^{\cH}(\rH\sfg^{-1}e_i,  e_i)$ by \cref{eq:traceManipulation}.

If $\sigma(x)$ is a linear operator on $\cE$, such that $\rH\sigma\sigma^{\sfT}\sfg\rH = \ttH$ and $\sigma$ is smooth in $x$, then the lifted It{\=o} and Stratonovich equations are
  \begin{gather}
    dX_t =  -\frac{1}{2}\sum_{i=1}^n\Gamma^{\cH}(e_i, \rH\sfg^{-1}e_i) dt + \rH\sigma dW_t,
    \label{eq:HorizontalIto}\\
dX_t = -\frac{1}{2}\sum_{i=1}^n(\nabla^{\cH}_{\rH\sigma e_i}\rH\sigma e_i)_{X_t} + (\rH\sigma)_{X_t}\circ dW_t, \label{eq:HorizontalStratAlt}\\
    dX_t = -\frac{1}{2}\sum_{i=1}^n\rH\Gamma^{\cH}(e_i, \rH\sfg^{-1}e_i)_{X_t} + \rH_{X_t}\circ (\sigma dW_t),
    \label{eq:HorizontalStrat}    
\end{gather}
where we assume the stronger condition $\rH\sigma\sigma^{\sfT} = \rH\sfg^{-1}$ for the last equation.
\end{theorem}
\section{Laplace-Beltrami operators and Brownian motions on matrix manifolds}\label{sec:examples}
We now consider matrix manifolds, where $\cE=\R^{n\times m}$ is a matrix vector space ($n, m$ are two positive integers). We equip $\cE$ with the trace (Frobenius) inner product $\Tr A^{\sfT}A$ for $A\in \cE$ to make it an inner product space. Note that we use the subscript $_{ij}$ to denote the $ij$-th entry of a matrix (a scalar), but we will use the notation $E_{ij}$ to denote an elementary matrix which is zero everywhere except that the $ij$-th entry is $1$, $\{E_{ij}\}_{i=1,j=1}^{i=n,j=m}$ forms the standard basis of $\cE$. We apologize for the possible confusion, however, we will not have a matrix named $E$. Thus, for a matrix $C$, we have $C = \sum_{ij} C_{ij}E_{ij}$ where the indices run over the dimensions of $C$. We have a useful identity
\begin{equation}C =\sum_{ij}E_{ij}C^{\sfT}E_{ij}\label{eq:useful}
\end{equation}
as $\sum_{ij}E_{ij}C^{\sfT}E_{ij} = \sum_{ijab} C_{ba}E_{ij}E_{ab}E_{ij}$, and $E_{ij}E_{ab}E_{ij}=0$ unless $a=j, b=i$, thus, the sum is $\sum_{ij} C_{ij}E_{ij}=C$. Alternatively, $C = C\sum_iE_{ii}=\sum_{ij}CE_{ij}E_{ij}=\sum_{ij}E_{ij}C^{\sfT}E_{ij}$ from \cref{eq:traceManipulation}.

Consider a function $f$ on $\cM$, extended to a function on $\cE$. In an SOO $(\ttA, \ttM)$, $\ttA$ and the Euclidean gradient $\egrad_f$ are matrices, $\ttM$ and the Euclidean Hessian $\ehess_f$ are operators on $\cE$. The trace of $\Pi\sfg^{-1}\ehess_f$ is the operator trace of the operator $\omega \mapsto \Pi\sfg^{-1}\ehess_f\omega$, which is $\sum_{ij}(\Pi\sfg^{-1}\ehess_fE_{ij})_{ij}$. 
\subsection{Matrix Lie groups with left-invariant metrics}
Brownian motions on a Lie group with a left-invariant metric has been well-studied. This section offers new formulas for the Stratonovich and It{\=o} drifts, as well as formulas for the projection and the Christoffel function discussed previously. They allow us to compute explicitly the drifts for common Lie groups, and also describe the Brownian motions explicitly.

We will consider a connected Lie subgroup $\rG$ of $\GL(N)$, the Lie group of invertible matrices, with Lie algebra $\fG\subset \R^{N\times N}$. The trace metric (trace inner product) $(A, B)\mapsto \sum_{ij}A_{ij}B_{ij} = \Tr A^{\sfT}B$ for $A, B\in\R^{N\times N}$ induces an inner product on $\fG$. A left-invariant metric on $\rG$ corresponds to an inner product on $\fG$, given by a positive-definite linear operator $\cI$ on $\fG$. We could define $\cI$ by fixing an orthonormal basis $\{v_i\}_{i=1}^{\dim\fG}$ of $\fG$ under the trace inner product, choosing a positive-definite matrix $\cI_P$ in $\R^{\dim\fG\times\dim\fG}$, and  then define $\cI v = \sum_{ij=1}^{\dim\fG}\cI_{P}^{ij} \langle v, v_i \rangle_{\cE}v_j$ for the inner product $\langle v, v\rangle_{\cI} = \sum_{ij=1}^{\dim\fG}\cI_{P}^{ij} \langle v, v_i \rangle_{\cE}\langle v, v_j \rangle_{\cE}$.

Under the trace metric on $\cE=\R^{N\times N}$, let $\pg$ be the orthogonal projection to $\fG$, 
(thus $\pg$ is self-adjoint in this metric), and we split $\cE =\fG\oplus \fGperp$, where $\fGperp = (\dI_{\cE} - \pg)\cE$. We extend $\cI$ to $\cE$ by choosing an inner product on $\fGperp$, while keeping $\fG$ and $\fGperp$ orthogonal in the $\cI$-pairing. We denote the extended operator on $\cE$ also by $\cI$.
The operation of $\cI$ on $\fGperp$ could be taken to be the identity operator, for example, the particular extension is not important. 
We define the following pairing on $\cE$, extending the pairing on $\fG$
\begin{equation}\langle \omega_1,\omega_2\rangle_{\cI} = \Tr\{\omega_1^{\sfT}\cI(\omega_2)\} = \Tr\{\cI(\omega_1)^{\sfT}\omega_2\}.
\end{equation}
\begin{lemma}Assume $\cI$ is self-adjoint in the trace metric and respects the orthogonality of $\fG$ and $\fGperp$, thus, $\cI\fG=\fG$ and $\cI\fGperp=\fGperp$ then $\pg$ {\it is also} the metric-compatible projection from $\cE$ to $\fG$ in the metric $\langle .,.\rangle_{\cI}$. That means
\begin{equation}\cI(\pg\omega) = \pg\cI(\omega).\end{equation}
\end{lemma}
\begin{proof}We verify $\nu^{\sfT}\cI(\pg\omega)= \nu^{\sfT}\pg(\cI(\omega))$ for $\nu\in\cE$. When $\nu\in \fGperp$, this is obvious as both sizes are zero by assumption. In the remaining case, $\nu\in\fG$, $\nu^{\sfT}\pg(\cI(\omega))= \nu^{\sfT}\cI(\omega)$ since $\pg$ is the orthogonal projection under trace, while 
$$\nu^{\sfT}\cI(\pg\omega)=(\cI\nu)^{\sfT}\pg\omega= (\cI\nu)^{\sfT}\omega = \nu^{\sfT}\cI(\omega)
$$
where the first equality is because $\cI$ is self-adjoint, the second is because $\cI\nu\in\fG$, and the third is because $\cI$ is self-adjoint again.

Since $\pg$ is self-adjoint under trace inner product, this means $\nu^{\sfT}\cI(\pg\omega)= \nu^{\sfT}\pg(\cI(\omega))= (\pg\nu)^{\sfT}\cI(\omega)$ for $\nu\in\cE$, thus $\pg$ is metric-compatible under $\cI$.
\end{proof}
A matrix $\xi\in \cE$ is in the tangent space $T_x\rG$ precisely if $x^{-1}\xi\in\fG$. Consider the left-invariant metric on $\rG$ induced by the pairing below on $\cE$ at $x\in\rG$
\begin{equation}(\omega_1, \omega_2)\mapsto \Tr \{(x^{-1}\omega_1)^{\sfT}\cI(x^{-1}\omega_2)\}\end{equation}
for $\omega_1,\omega_2\in\cE$, or equivalently, for $\omega\in\cE$, define the metric operator
\begin{equation}\sfg(x)(\omega) :=  (x^{-1})^{\sfT}\cI(x^{-1} \omega).\label{eq:gmetric}
\end{equation}
\emph{We will use Einstein's summation convention for the remaining sub-sections of this section, unless indicated otherwise. }
\begin{theorem}Under the metric \cref{eq:gmetric} defined by $\cI$, for $x\in\rG, \omega\in\cE$
\begin{gather}
  \Pi(x)(\omega) = x(x^{-1}\omega)_{\pg}\label{eq:Gproj},\\
  \Pi(x)\sfg(x)^{-1}(\omega) = x\cI^{-1}(x^{\sfT} \omega)_{\pg}, \label{eq:Gprojg}
\end{gather}
\begin{equation}
\begin{split}
  \Gamma(x; \xi, \eta) &=   - \frac{1}{2}(\xi x^{-1}\eta  + \eta x^{-1}\xi) \\
            &\quad+ \frac{1}{2}x \cI^{-1}\{
            [\cI(x^{-1}\xi), (x^{-1}\eta)^{\sfT}] \
            + [\cI(x^{-1}\eta), (x^{-1}\xi)^{\sfT}]\}_{\pg},\label{eq:GammaG}
\end{split}
\end{equation}
where $\lbrack,\;\rbrack$ is the Lie bracket. Using Einstein's summation convention, we have
\begin{equation}  
\begin{split}
  (\Delta_f)_{x} &=   \{x(\cI^{-1}(x^{\sfT} \ehess_f E_{ij}))_{\pg}\}_{ij} \\
  &\quad+ \Tr\egrad_f^{\sfT}x\{
  E_{ij}\cI^{-1}(E_{ij})_{\pg}  - \cI^{-1}[(E_{ij})_{\pg}, E_{ji}]_{\pg}\}.
\end{split}  
\end{equation}
Let $B^{[N\times N]}$ be a Wiener process on $\R^{N\times N}$. Then $(B^{[N\times N]})_{\pg}$ is a Euclidean Brownian motion on $\fG$ with the trace inner product. The Stratonovich and It{\=o} form of the Riemannian Brownian motion on $\rG$ are given by
\begin{gather}
  dX_t = X_t\circ \cI^{-\frac{1}{2}}(dB^{[N\times N]})_{\pg} -\frac{1}{2} X_t\cI^{-1}
  [(E_{ij})_{\pg}, E_{ji}]_{\pg}dt,\label{eq:matrixStrat}\\  
  dX_t = X_t \cI^{-\frac{1}{2}}(dB^{[N\times N]})_{\pg} + \frac{1}{2}X_t\{
  E_{ij}\cI^{-1}(E_{ij})_{\pg}  - \cI^{-1}[(E_{ij})_{\pg}, E_{ji}]_{\pg}\} dt.\label{eq:matrixIto}
\end{gather}
\end{theorem}
It is known \cite{liao2004levy} if $v_{\alpha}$, $\alpha=1,\cdots, \dim\fG$ is an orthonormal base of $\fG$ in the left invariant metric $\cI$, and $v_0 = -\frac{1}{2}\langle v_{\beta}, [v_{\alpha}, v_{\beta}]\rangle_{\sfg}v_{\alpha}$, then the Stratonovich equation for the Brownian motion on $\rG$ is of the form
\begin{equation}
  dX_t = X_t\circ dW^{\alpha}v_{\alpha} + X_tv_0 dt
\end{equation}
where $W$ is a Wiener process on $\R^{\dim\fG}$. The stochastic component in \cref{eq:matrixStrat} is easily seen to be consistent with this. We will show the drift is also consistent. Our formula shows the Stratonovich drift $\mu_{\cI}$ corresponding to $\cI$ is $\mu_{\cI}=X_t\cI^{-1}(X_t^{-1}\mu_{\dI_{\cE}})$, thus, if the drift is zero for one metric then it is zero for all others, as is well-known. The drift is zero for unimodular groups. The condition $v_0=0$ implies $\Tr\ad_{v_{\alpha}}=0$ for each basis vector $v_{\alpha}$, hence $\Tr\ad_\xi=0$ for each $\xi\in\fG$, the infinitesimal condition for unimodularity. Thus, we can check unimodularity with \cref{eq:matrixStrat}.
\begin{proof}It is clear $\Pi$ in \cref{eq:Gproj} satisfies $x^{-1}\Pi(x)\omega \in\sfg$, hence $\Pi(x)\omega\in T_x\rG$, and $\Pi(x)^2=\Pi$. Metric compatibility follows from left-invariance. As $\sfg(x)^{-1}\omega = x\cI^{-1}(x^{\sfT}\omega)$, \cref{eq:Gprojg} follows. To derive $\Gamma$, we can start with \cref{eq:GammaLeviCivita} and simplify the expression, but it is easier to show $\Gamma$ given in \cref{eq:GammaG} satisfies the requirements of the Christoffel function as reviewed in \cref{subsec:Levi}. Thus, consider $\Gamma$ as defined in \cref{eq:GammaG}. Since the torsion-free requirement $\Gamma(\ttX, \ttY) = \Gamma(\ttY, \ttX)$ is immediate if $\ttX, \ttY$ are two vector fields, it remains to verify $\nabla_{\ttX}\ttY$ in \cref{def:Levi} is a vector field, and \cref{eq:MetricComp} is satisfied. Setting $v(x) := x^{-1}\ttX(x)\in \fG$ and $w(x):=x^{-1}\ttY(x)\in \fG$, we want to show
$x^{-1}(\rD_{\ttX}\ttY + \Gamma(\ttX, \ttY))_x\in \fG$. As the second line of \cref{eq:GammaG} is in $x\fG$, it suffices to show for $\ttX(x) = xv(x), \ttY(x) = xw(x)$
$$\begin{gathered}x^{-1}\rD_{\ttX}\ttY -\frac{1}{2}(x^{-1}\ttX x^{-1}\ttY + x^{-1}\ttY x^{-1}\ttX)\in \fG\\
  \Leftrightarrow v(x)w(x) + \rD_{xv(x)}w(x)-\frac{1}{2}(v(x)w(x)+w(x)v(x))\in \fG\\
    \Leftrightarrow \rD_{xv(x)}w(x) +\frac{1}{2}[v(x), w(x)]\in\fG.
\end{gathered}$$
The last condition holds since $w$ is a function from $\rG$ to $\fG$, thus $\rD_{xv}w\in\fG$, while $[v, w]\in\fG$ as $\fG$ is a Lie-subalgebra (we drop the variable name for brevity). For metric compatibility, we expand two expressions below
$$\begin{gathered}\rD_{\ttX}\Tr \{(x^{-1}\ttY)^{\sfT}\cI(x^{-1}\ttY)\} =
  \rD_{xv}\Tr w^{\sfT}\cI(w)= 2\Tr(\rD_{xv}w)^{\sfT}\cI(w),
  \\
2\langle \ttY(x), \nabla_{\ttX}\ttY \rangle_{\sfg} 
=2\Tr(w^{\sfT}\cI(\rD_{xv}w +\frac{1}{2}[v, w]
+\frac{1}{2} \cI^{-1}\{
            [\cI(v), w^{\sfT}]  + [\cI(w), v^{\sfT}]\}_{\pg}).
\end{gathered}$$
In computing the difference, since $\cI$ is self-adjoint under the trace inner product, the terms with $\Tr(\rD_{xv}w)^{\sfT}\cI(w)$ cancels. It remain to simplify the below using $w\in \fG$  and $\Tr A[B,C] = \Tr B[C,A]$
$$\begin{gathered}
  \Tr(w^{\sfT}\{\cI([v, w]) + [\cI(v), w^{\sfT}]_{\pg} + [\cI(w), v^{\sfT}]_{\pg}\})\\
  =\Tr\cI(w)^{\sfT}[v, w] + \Tr w^{\sfT}[\cI(v), w^{\sfT}] + \Tr w^{\sfT}[\cI(w), v^{\sfT}]\\
  =\Tr\cI(w)^{\sfT}[v, w]  + \Tr \cI(w)[v^{\sfT}, w^{\sfT}]
  =\Tr\cI(w)^{\sfT}[v, w]  + \Tr \cI(w)[w, v]^{\sfT}
  = 0
  \end{gathered}$$
since $\Tr w^{\sfT}[\cI(v), w^{\sfT}]=0$, $\Tr w^{\sfT}[\cI(w), v^{\sfT}] = \Tr\cI(w)[v^{\sfT},w^{\sfT}]$ between the second and the last line. Thus, $\nabla$ is metric invariant. Note that $\Gamma$ is left invariant, that means $x^{-1}\Gamma(x; \xi, \eta) = \Gamma(\dI_{\cE}; x^{-1}\xi, x^{-1}\eta)$. 

For $\Delta_f$ as an SOO of the form $(\ttA, \ttM)$, we compute
$$x^{-1}\ttA= -x^{-1}\Gamma(x; E_{ij}, x\cI^{-1}(x^{\sfT}E_{ij})_{\pg})=
-\Gamma(\dI_{\cE}; x^{-1}E_{ij}, \cI^{-1}(x^{\sfT}E_{ij})_{\pg})
$$
which is $-\Gamma(\dI_{\cE}; E_{ij}, \cI^{-1}(E_{ij})_{\pg})$ by \cref{eq:traceManipulation}. Simplify further
$$\begin{gathered}
x^{-1}\ttA= \frac{1}{2}(E_{ij} \cI^{-1}((E_{ij})_{\pg})  + \cI^{-1}((E_{ij})_{\pg}) E_{ij}) \\
            - \frac{1}{2} \cI^{-1}\{
            [\cI(E_{ij}), (\cI^{-1}((E_{ij})_{\pg}))^{\sfT}] \
            + [\cI(\cI^{-1}((E_{ij})_{\pg})), (E_{ij})^{\sfT}]\}_{\pg}.
\end{gathered}            $$
Using \cref{eq:traceManipulation} again, the first line simplifies to $E_{ij} \cI^{-1}(E_{ij})_{\pg}$. The second line also simplifies to
$$-\frac{1}{2} \cI^{-1}\{
[E_{ij}, (\cI^{-1}(\cI(E_{ij})_{\pg}))^{\sfT}]
+ [(E_{ij})_{\pg}, (E_{ij})^{\sfT}]\}_{\pg}
= -\cI^{-1}[(E_{ij})_{\pg}, E_{ji}]_{\pg}.
$$
From here, we get $\Delta_f$ and the It{\=o} drift. Note that we can take $\sigma(x)\omega = x\cI^{-\frac{1}{2}}(\omega_{\pg})$, with $\sigma^{\sfT}(x)\omega = \cI^{-\frac{1}{2}}(x^{\sfT}\omega)_{\pg}$ verifying $\sigma\sigma^{\sfT} = \Pi\sfg^{-1}$. Thus,
$$\rD_{\sigma(x) E_{ij}}\sigma(x) E_{ij} = x\cI^{-\frac{1}{2}}(E_{ij})_{\pg} \cI^{-\frac{1}{2}}(E_{ij})_{\pg}=xE_{ij} \cI^{-1}(E_{ij})_{\pg},
$$
which gives us \cref{eq:matrixStrat}. Finally, write $\xi =\langle v_{\alpha}, \xi\rangle_{\sfg}v_{\alpha}$ for $\xi\in\fG$, thus, for fixed $\beta$, $[\xi, v_{\beta}] = \langle v_{\alpha}, \xi\rangle_{\sfg}[v_{\alpha},v_{\beta}]$. Summing over $\alpha$ and $\beta$ and applying \cref{eq:Bsum} 
$$\begin{gathered}\langle v_{\beta}, [v_{\alpha}, v_{\beta}]\rangle_{\sfg}\langle v_{\alpha},\xi\rangle_{\sfg}
=\langle v_{\beta},[\xi, v_{\beta}]\rangle_{\sfg}=
\Tr \{\cI(v_{\beta})^{\sfT} [\xi, v_{\beta}]\}\\
=\Tr \{(\cI(\Pi\sfg^{-1}(E_{ij}))^{\sfT} [\xi, E_{ij}]\}
=\Tr \{(\cI(\cI^{-1}(E_{ij})_{\pg}))^{\sfT} [\xi, E_{ij}]\}\\
=\Tr \{((E_{ij})_{\pg})^{\sfT} [\xi, E_{ij}]\}
=\Tr \{\xi  [ E_{ij}, ((E_{ij})_{\pg})^{\sfT}]\}= \langle \cI^{-1}[ (E_{ij})_{\pg}, E_{ji}]_{\pg},\xi\rangle_{\sfg}.
\end{gathered}
$$
Since the inner product is nondegenerate on $\fG$, $\langle v_{\beta}, [v_{\alpha}, v_{\beta}]\rangle_{\sfg} v_{\alpha} = \cI^{-1}[ (E_{ij})_{\pg}, E_{ji}]_{\pg}$. This confirms the alternative form of the Stratonovich drift.
\end{proof}
Let us apply this result to Brownian motions of a few matrix groups. As mentioned, the Stratonovich drift for an arbitrary $\cI$ could be computed from the case $\cI=\dI_{\cE}$. The It{\=o} drift in general can not be simplified, but we will compute for $\cI=\dI_{\cE}$ also.

For $\GL^+(N)$, the group of $\R^{N\times N}$ matrices with positive determinant, $\pg$ is the identity map, $[(E_{ij})_{\pg},  E_{ji}]_{\pg}= E_{ii}- E_{jj}=\dI_N-\dI_N =0$, thus the Stratonovich drift is zero, as well-known. The It{\=o} drift is $\frac{1}{2}X_tdt$.

For $\SL(N)$, the subgroup of $\GL^+(N)$ with determinant $1$, $\pg(\omega) = \omega -\frac{\Tr\omega}{N}\dI_N$, for $\omega\in\R^N$, and $[E_{ij} - \frac{\Tr E_{ij}}{N}\dI_N, E_{ji}]_{\pg} = 0$ (since 
$[\dI_N, E_{ji}]=0$ while $[E_{ij}, E_{ji}]=0$ above), so the Stratonovich drift is also zero. The It{\=o} drift is $\frac{1}{2}X_t(E_{ij}(E_{ij}-\frac{\Tr E_{ij}}{N}\dI_N))dt = \frac{N-1}{2N}X_tdt$.

Consider $\Aff(N)$, the subgroup of $\GL^+(N+1)$ of the form $\begin{bmatrix}A& v\\0_{1\times N} & 1\end{bmatrix}$ with $A\in \GL^+(N), v\in\R^{N}$. This group could be used to model an object which could both move and deform linearly in $\R^N$, with the vector $v$ model an anchor point on the object, while $A$ specifies the linear deformation based at the anchor point. Here, $\fG$ is the subspace of $\R^{(N+1)\times(N+1)}$ with the last row is zero, $\pg(\omega)$ sends the last row of $\omega\in\R^{(N+1)\times(N+1)}$ to zero, hence
  $$\begin{gathered}\lbrack(E_{ij})_{\pg}, E_{ji}\rbrack_{\pg}=\sum_{i=1}^N[E_{i,N+1}, E_{N+1,i}]_{\pg} =\begin{bmatrix}\dI_N& 0_{N\times 1}\\0_{1\times N} & 0\end{bmatrix}=: J_N,\\
  E_{ij}(E_{ij})_{\pg} = \begin{bmatrix}\dI_N& 0_{N\times 1}\\0_{1\times N} & 0\end{bmatrix}.
  \end{gathered}$$
Thus, the Stratonovich drift is $-\frac{1}{2}X_t\cI^{-1}(J_N)$, the It{\=o} drift for $\cI=\dI_{\cE}$ is zero.

For $\SO(N)$, the subgroup of orthogonal matrices of determinant $1$, $\fG$ is the group of antisymmetric matrices and $\omega_{\pg} = \omega_{\asym}$. We have
$$\begin{gathered}\lbrack(E_{ij})_{\pg}, E_{ji}\rbrack_{\pg}=\frac{1}{2} \lbrack E_{ij}, E_{ji}\rbrack_{\asym} - \frac{1}{2} \lbrack E_{ji}, E_{ji}\rbrack_{\asym}
=0,\\
  E_{ij}(E_{ij})_{\asym} = \frac{1}{2}E_{ij}E_{ij} - \frac{1}{2}E_{ij}E_{ji}=
  \frac{1-N}{2}\dI_N.
\end{gathered}$$
Thus, the Stratonovich drift is zero, the It{\=o} drift for $\cI=\dI_{\cE}$ is $\frac{1-N}{4}X_t dt$, agreeing with \cite{LaplaceEmb}. In the case where $\cI$ is given by an $N\times N$ symmetric matrix $\bar{\cI}$ with positive entries such that $\cI(E_{ij}) = \bar{\cI}_{ij}E_{ij}$, the It{\=o} drift is given by $-\frac{1}{4}X_t\sum_i(\sum_{j\neq i}\bar{\cI}_{ij}^{-1})E_{ii}dt$. 

The special Euclidean group $\SE(N)$ allows us to simulate Brownian movements of rigid bodies. $\SE(N)$ is the subgroup of $\Aff(N)$ with $A$ orthogonal, $\pg$ sends the top $N\times N$ block to its antisymmetry part, and zero out the last line. Thus, the Stratonovich drift is zero
$$\begin{gathered}\lbrack(E_{ij})_{\pg}, E_{ji}\rbrack_{\pg}=\frac{1}{2} \sum_{i\leq N, j\leq N}\lbrack E_{ij}, E_{ji}\rbrack_{\asym}+\sum_{i\leq  N}[E_{i, N+1}, E_{N+1, i}]_{\asym} =0.
\end{gathered}$$
The difference between this and $\Aff(N)$ is $[(E_{ij})_{\pg}, E_{ji}]$ usually gives a diagonal block, which is zeroed out by $\pg$ in this case but not in the affine case. For a matrix group to be not unimodular, $[(E_{ij})_{\pg}, E_{ji}]$ must not be zero, and heuristically, the Lie algebra should contain some diagonal blocks.
\subsection{Positive-definite matrices with the affine invariant metric} The manifold $\cM = \Sd{N}$ consists of positive definite symmetric matrices in $\cE=\R^{N\times N}$ with the metric operator $\sfg(x) \omega =x^{-1}\omega x^{-1}, x\in\Sd{N}$. This defines a pairing on $\cE^2$ which restricts to the {\it affine invariant metric} on $\Sd{n}$  
\begin{equation}\langle \omega, \omega\rangle_{\sfg} = \Tr \omega^{\sfT}x^{-1}\omega x^{-1}.
  \label{eq:Sdmetric}
\end{equation}
We have $\Pi\omega = \omega_{\sym}$ and $\Gamma(x, \xi, \eta) = -(\xi x^{-1}\eta )_{\sym}$. Thus, we need to evaluate
$$\begin{gathered}\ttA = (E_{ij}x^{-1}(x E_{ij}x)_{\sym})_{\sym}
  = \frac{1}{2}(E_{ij}E_{ij}x + E_{ij}E_{ji}x  )_{\sym}=\frac{N+1}{2}x,
\end{gathered}$$
\begin{proposition}The Laplace-Beltrami operator of $\Sd{n}$ with the metric \cref{eq:Sdmetric} at $x\in\Sd{n}$ is given by
  \begin{equation}\Delta_f = \frac{N+1}{2}\Tr x \egrad_f + (\ehess_f(x E_{ij} x)_{\sym})_{ij}.\label{eq:SdLaplace}
  \end{equation}
Let $B^{[N\times N]}_t$ be a Wiener processes on $\R^{N\times N}$ and $\{\epsilon_{ij}|1\leq i\leq j\leq N\}$ be the standard basis of $\Herm{N}$, $\epsilon_{ij} =\begin{cases}\sqrt{2}(E_{ij})_{sym} \text{ for }1\leq i<j\leq N\\
    E_{ii} \text{ for } 1\leq i=j\neq N.\end{cases}$. Expressing $(B^{[N\times N]}_t)_{\sym}$ in this basis we get a Wiener process $W_t$ on $\R^{\frac{N(N+1)}{2}}$. Define
  \begin{equation}\sigma(x) \omega = x^{\frac{1}{2}}\omega x^{\frac{1}{2}}
  \end{equation}
for $\omega \in \R^{N\times N}$, where $x^{\frac{1}{2}}$ is the positive-definite matrix square root. Then $\sigma$ satisfies $\Pi\sigma\sigma^{\sfT}\sfg\Pi = \Pi$, and can be used to construct the Riemannian Brownian motion in It{\=o} form in \cref{eq:BrownianIto}
\begin{equation}dX_t= \frac{N+1}{4}X_t dt + X_t^{\frac{1}{2}} (dB^{[N\times N]}_t)_{\sym}X_t^{\frac{1}{2}} = \frac{N+1}{4}X_t dt + \sigma(X_t) (\sum_{i\leq j}dW^{ij}_t\epsilon_{ij}).\label{eq:SdIto}
\end{equation}
Let $x = V\diag(\beta_1^2,\cdots, \beta_n^2)V^{\sfT}$ be a symmetric eigenvalue decomposition. Set
\begin{equation}S(x) := -V\diag(\beta_1^2(\frac{1}{4} +\sum_j\frac{\beta_j}{2(\beta_1+\beta_j)})),\cdots,\beta_i^2(\frac{1}{4} +\sum_j\frac{\beta_j}{2(\beta_i+\beta_j)})),\cdots)V^{\sfT}.
\end{equation}
Then the Stratonovich form in \cref{eq:BrownianStratAlt} is given by
\begin{equation}dX_t= (S(X_t)+ \frac{N+1}{4}X_t )dt + \sigma(X_t)\circ (\sum_{i\leq j}dW^{ij}_t\epsilon_{ij}).\label{eq:SdStratAlt}
\end{equation}
\end{proposition}
We only mentioned the Stratonovich equation (\ref{eq:BrownianStratAlt}) since the form \cref{eq:BrownianStrat} is the same as the It{\=o} form, as $\Pi$ is constant.
\begin{proof}The expression in \cref{eq:SdLaplace} for the Laplace operator follows from the discussion preceding the proposition.

It is clear $\sigma^{\sfT} = \sigma$ and $\sigma\sigma^{\sfT}=\sfg^{-1}$, thus the It{\=o} form is as in the middle expression of \cref{eq:SdIto}. Since $\sym$ is the projection from $\R^{N\times N}$ to $\Herm{n}$, $(B^{[N\times N]}_t)_{\sym}$ is a Wiener process in $\Herm{N}$ (see \cite{Hsu}).
  For \cref{eq:SdStratAlt}, we need to evaluate (with Einstein's summation convention)
  $\rD_{\sigma(x)(E_{ij})_{\sym}}\sigma(x)(E_{ij})_{\sym}$.
  
  By differentiating the equation $(x^{\frac{1}{2}})^2 = x$, the Fr{\'e}chet derivative of the square root satisfies $\rD_\omega (x\mapsto x^{\frac{1}{2}}) = \fL_{x^{\frac{1}{2}}}^{-1}\omega$, where $\fL_A(\omega) = A\omega + \omega A$ is the Lyapunov operator for a symmetric matrix $A\in\Herm{n}, \omega\in\R^{N\times N}$. If $A = VDV^{\sfT}$ is a symmetric eigenvalue decomposition,  $D=\diag(d_1,\cdots,d_N)$ such that $d_i+d_j\neq 0$ for all $i,j$ and $V^{\sfT}V=\dI_N$, then denote by $\odot$ the by-entry multiplication (Hadamard product), with $\hat{D}=(\hat{D}_{ij})_{i,j=1}^N$ below, we have
  $$\fL_A^{-1}B = V((V^{\sfT}BV)\odot\hat{D}  )V^{\sfT}; \hat{D}_{ij} = \frac{1}{d_i+d_j}\text{ for }i,j=1\cdots N.
  $$
Thus, if $x = VD^2 V^{\sfT}$, $D = \diag(\beta_1,\cdots, \beta_n)$ and write $\sigma_x$ for $\sigma(x)$
  $$\rD_{\sigma_x(E_{ij})_{\sym}}\sigma_x(E_{ij})_{\sym}=
2 ((\fL_{x^{\frac{1}{2}}}^{-1} \sigma_x(E_{ij})_{\sym})(E_{ij})_{\sym}x^{\frac{1}{2}} )_{\sym}.
$$
Expand $2(\fL_{x^{\frac{1}{2}}}^{-1} \sigma_x(E_{ij})_{\sym})(E_{ij})_{\sym}x^{\frac{1}{2}}=VRV^{\sfT}$ with the summation convention, where we use \cref{eq:traceManipulation} to remove $\sym$ then $V$
$$\begin{gathered}
R = 2\{(V^{\sfT}(VDV^{\sfT}) (E_{ij})_{\sym}(VDV^{\sfT})V)\odot\hat{D} \}V^{\sfT}(E_{ij})_{\sym}VD\\
=\{(DE_{ij}D)\odot\hat{D} \}(E_{ij}+E_{ji})D
=(E_{ij}\frac{\beta_i\beta_j}{\beta_i+\beta_j} (\beta_jE_{ij}+\beta_i E_{ji})\\
=\frac{1}{2}\beta_i^2E_{ii} + \frac{\beta_i^2\beta_j}{\beta_i+\beta_j}E_{ii}.
\end{gathered}$$
From here, we get the expression for $S(X_t)$ and \cref{eq:SdStratAlt}.
\end{proof}  
\subsection{Stiefel manifolds}\label{subsec:stiefel}
For positive integers $p <n$, consider the manifold $\St{n}{p}\subset \cE=\R^{n\times p}$ of orthogonal matrices $Y\in \R^{n\times p}$, $Y^{\sfT}Y = \dI_p$. We equip $\cE$ with the base inner product $\langle \omega_1, \omega_2\rangle_{\cE} = \Tr\omega_1^{\sfT}\omega_2 $. For $\alpha_0, \alpha_1 > 0$, define the metric operator
\begin{equation}\sfg:\omega \mapsto \alpha_0(\dI_n-YY^{\sfT})\omega + \alpha_1 YY^{\sfT}\omega
=\alpha_0K_0\omega + \alpha_1K_1\omega\label{eq:stiefmetric}
\end{equation}
and the associated metric $\langle \omega,\omega\rangle_{\sfg} = \alpha_0\Tr\omega^{\sfT} K_0\omega + \alpha_1\Tr\omega^{\sfT}K_1\omega$, where $K_0 = \dI_n-YY^{\sfT}, K_1 = YY^{\sfT}$. From \cite{ExtCurveStiefel,NguyenOperator}, $\sfg^{-1} \omega = \alpha_0^{-1}K_0\omega + \alpha_1^{-1} K_1\omega$ and
\begin{gather}
  \Pi\omega = \omega - Y(Y^{\sfT}\omega)_{\sym},\\
  \Gamma(Y, \xi, \eta) = Y(\xi^{\sfT}\eta )_{\sym} +2\frac{\alpha_0 - \alpha_1}{\alpha_0}K_0(\xi\eta^{\sfT})_{\sym}Y.\label{eq:stiefelGamma}
\end{gather}
Write $\omega = K_0\omega + K_1\omega$, then $\Pi (K_0\omega) = K_0\omega$, $\Pi (K_1\omega) = Y(Y^{\sfT}\omega)_{\asym}$. Note $\Tr K_1= \Tr Y^{\sfT}Y = p, \Tr K_0 = \Tr (\dI_n-K_1) = n-p$. To evaluate $\ttA$, we compute the sum $\Gamma(E_{ij}, \Pi\sfg^{-1}E_{ij})$ in parts (with Einstein's summation convention)
$$\begin{gathered} E_{ij}^{\sfT}(\Pi \sfg^{-1}E_{ij}) = E_{ij}^{\sfT}(\alpha_0^{-1}K_0E_{ij} + \alpha_1^{-1}Y(Y^{\sfT}E_{ij})_{\asym})\\
  =\alpha_0^{-1}E_{ij}^{\sfT}K_0E_{ij} + \alpha_1^{-1}E_{ij}^{\sfT}Y(Y^{\sfT}E_{ij})_{\asym}\\
  =\frac{(K_0)_{ii}}{\alpha_0}E_{jj} + \frac{1}{2\alpha_1}(E_{ij}^{\sfT}YY^{\sfT}E_{ij}-E_{ij}^{\sfT}YE_{ji}Y)   \\
  =\frac{\Tr K_0}{\alpha_0}\dI_p + \frac{1}{2\alpha_1}(\Tr(YY^{\sfT})\dI_p - Y^{\sfT}Y)
=(\frac{n-p}{\alpha_0} + \frac{p-1}{2\alpha_1})\dI_p,
\end{gathered}$$
$$\begin{gathered}
  E_{ij}(\Pi \sfg^{-1}E_{ij})^{\sfT}_{\sym} =
   E_{ij}(\alpha_0^{-1}E_{ij}^{\sfT}K_0 - \alpha_1^{-1}(Y^{\sfT}E_{ij})_{\asym}Y^{\sfT})\\
   =\frac{p}{\alpha_0}K_0 - \frac{1}{2\alpha_1}(E_{ij}Y^{\sfT}E_{ij}Y^{\sfT} - E_{ij}E_{ij}^{\sfT}YY^{\sfT})
=\frac{p}{\alpha_0} K_0 + \frac{p-1}{2\alpha_1}YY^{\sfT}
\end{gathered}$$
where we use \cref{eq:useful}. Thus, $K_0(E_{ij}(\Pi \sfg^{-1}E_{ij})^{\sfT}_{\sym})Y=0$ and
$$\ttA = - \sum_{ij}\Gamma(Y; E_{ij},  \Pi \sfg^{-1}E_{ij})=
-(\frac{n-p}{\alpha_0} + \frac{p-1}{2\alpha_1})Y.
$$
and the Laplacian is
\begin{equation}-(\frac{n-p}{\alpha_0} + \frac{p-1}{2\alpha_1})\Tr Y^{\sfT}\egrad_f +
  (\ehess_f (\Pi\sfg^{-1}E_{ij}))_{ij}.
\end{equation}
For the Brownian motion, we can take $\sigma\omega=\sfg^{-\frac{1}{2}}\omega =\alpha_0^{-\frac{1}{2}}K_0\omega + \alpha_1^{-\frac{1}{2}} K_1\omega$ with an ambient process $W_t$ in $\R^{n\times p}$
\begin{equation}dY_t = -(\frac{n-p}{2\alpha_0} + \frac{p-1}{4\alpha_1})Y_tdt + \{  \alpha_0^{-\frac{1}{2}}(\dI_n -Y_tY_t^{\sfT}) dW_t + \alpha_1^{-\frac{1}{2}} Y_t(Y_t^{\sfT} dW_t)_{\asym}\}.
\end{equation}  
Note the formula reduces to the sphere case when $p =1, \alpha_0 = 1$.

\subsection{Grassmann manifolds}Consider the manifold $\Gr{n}{p}$, the quotient of the Stiefel manifold $\St{n}{p}$ by the action of $\SO(p)$ on the right. This is the configuration space of $p$-dimension subspaces of $\R^n$. Recall the tangent space at $Y\in\St{n}{p}$ consists of matrices $\eta\in\R^{n\times p}$ satisfying $\sym(Y^{\sfT}\eta) = 0$. In the lift to $\St{n}{p}$, with the metric defined by $\alpha_0 = \alpha_1 = 1$, the horizontal space consists of tangent vectors $\eta$ satisfying $Y^{\sfT}\eta=0$, the horizontal projection is $\rH\omega = (\dI_n -YY^{\sfT})\omega$ and the Christoffel function for two horizontal vectors $\xi, \eta$ is evaluated as $\Gamma^{\cH}(\xi, \eta) = Y\xi^{\sfT}\eta$ \cite[eq. 2.71]{Edelman_1999}. From here
$$ \Gamma^{\cH}(E_{ij}, \rH\sfg^{-1}E_{ij}) = YE_{ij}^{\sfT}(\dI_n-YY^{\sfT})E_{ij}
= Y\Tr(\dI-YY^{\sfT}) = (n-p)Y.
$$
Using $\sigma = \dI_{\cE}$, to check the mean curvature condition on the submanifold $Y\SO(p)\subset \St{n}{p}$ at $Y\in \St{n}{p}$, for $A, B\in\oo(p)$, and we write $YA, YB$ for the corresponding vector fields. From \cref{eq:stiefelGamma}
$$\nabla_{YA}YB = YAB + Y(-AY^{\sfT}YB)_{\sym}=\frac{1}{2}Y[A,B]$$
is tangent to the orbit and we have the condition $\ttH\nabla_{\ttZ_i}\ttZ_i=0$ in \cref{sec:hBrown}.

The lifted Laplacian, and horizontal Brownian motion equations are
\begin{gather}\Delta^{\cH}_f = -(n-p)\Tr Y^{\sfT}\egrad_f + (\ehess_f(\dI-YY^{\sfT})E_{ij})_{ij},\\
  dY_t = -\frac{n-p}{2}Y_t + (\dI_n-YY^{\sfT})dW_t^{[n\times p]},\\
  dY_t =  (\dI_n-YY^{\sfT})\circ dW_t^{[n\times p]}.
\end{gather}
where for the last equation, we use (with Einstein's summation)
$$\rD_{\rH E_{ij}} YY^{\sfT}E_{ij} =  (\dI_n-YY^{\sfT})E_{ij}Y^{\sfT}E_{ij} + Y((\dI_n-YY^{\sfT}) E_{ij})^{\sfT}E_{ij}=(n-p)Y.$$
noting $(\dI_n-YY^{\sfT})E_{ij}Y^{\sfT}E_{ij}=(\dI_n-YY^{\sfT})Y=0$ by \cref{eq:useful}. In \cite{Baudoin}, the authors consider the complex Grassmann manifold, using a different method. We will leave it to a follow-up work to derive the It{\=o} and Stratonovich versions above for the complex Grassmann.
\section{Numerical methods: deterministic and stochastic projection methods}\label{sec:numeric}
We now discuss numerical methods to solve SDE on $\R^n$ with an invariant manifold $\cM$ (the SDE are called conservative). In the ODE literature, it is well known \cite[section VII.2]{HairerWarner}, \cite[section 5.3]{SoellnerFuhrer}, that a numerical scheme on $\R^n$, when applied to a conservative ODE together with a {\it projection} will in general preserve the order of accuracy. Projection in the numerical analysis literature is just the {\it nearest point retraction} in \cref{sec:tubular}. In the previous sections, we used the term projection for a linear projection to the tangent space of a manifold point. To distinguish, while we still use the term {\it projection method}, projections will be used only for these linear maps, otherwise, we will use the term retraction. While we have seen \footnote{See \href{https://mtaylor.web.unc.edu/wp-content/uploads/sites/16915/2020/08/expomap.pdf}{https://mtaylor.web.unc.edu/wp-content/uploads/sites/16915/2020/08/expomap.pdf}, (Professor M.E. Taylor's note on geodesics)} the projection method applied to the geodesic equation and several examples of projection method for SDE mentioned in \cref{sec:previous}, the linkage with SOO clarifies the condition when the stochastic version applies. We will present the results in \cref{subsec:projstochastic}.

It is known \cite[corollary 2.3]{SHSW} that we can simulate geometric random walks using a second-order tangent retraction (see \cref{eq:geo} below) with respect to the Levi-Civita connection, but there are only a few known examples of such retractions.  Inspired by the Taylor method applied to \cref{cor:ANPT}, in \cref{theo:geodesic}, we construct {\it a second-order tangent retraction} from any $\cE$-tubular retractions, for {\it any torsion-free connection}. In addition, we also introduce the {\it retractive Euler-Maruyama} method in \cref{sec:retractivesim}.

We show in \cref{alg:Simulation} the basic simulation framework for a $\cM$-valued process $X_t$ described by either the It{\=o} ($dX_t = \mu dt + \sigmam dW_t$) or Stratonovich ($dX_t = \mu dt + \sigmam \circ dW_t$) form on the ambient space $\cE$ with an invariant manifold $\cM$. The rest of the paper will discuss different integrators' $\fF$'s, and will provide numerical experiments using these integrators.
\begin{algorithm}
 \label{alg:RQI}
 \begin{algorithmic}[1]
 \State{Input: $X_0\in \cM$, an integrator $\fF$, $T, n_{div}, n_{path}$ , functions $g, f$, a SDE of It{\=o} or Stratonovich type for $X$ with drift $\mu$ and stochastic map $\sigmam$.}\;
 \State{$h\gets \frac{T}{n_{div}}$}\;
 \For{$i=0,\cdots, n_{path}-1$}\;
 \State $X_{i,0} \gets X_0$\;
 \State $s_i \gets g(X_{i,0}, 0)$\;
 \For{$j=1,\cdots, n_{div}-1$}\;
 \State Sample $\Delta\sim N(0, I_k)$
 \State Compute $X_{i,j+1}\gets \fF(\mu , \sigmam, X_{ij}, jh,\Delta, h)$\;\Comment{$X_{i,j+1}\in \cM$}
 \State Compute $s_i \gets s_i + g(X_{i,j+1}, jh)$\;
 \EndFor
 \State Compute $s_i \gets s_i + f(X_T, T)$\;
 \EndFor
 \State{ Return: The collection $\{s_i\}_{i=1}^{n_{path}}$. }\Comment{To compute statistics}
 \end{algorithmic}
 \caption{Simulating the distribution of $\int_0^Tg(X_s, s)ds + f(X_{T}, T)$ for a manifold valued-process $X_t$ satisfying an It{\=o} or Stratonovich SDE.}
 \label{alg:Simulation}
 \end{algorithm}

\subsection{Projection method and geodesic equation}\label{sec:projgeo}
Consider a manifold $\cM\subset \cE$, where $\cE$ is a vector space. Let $D_{\cM}$ be a {\it tubular neighborhood} of $\cM$ (\cref{sec:tubular}). Consider the differential equation
\begin{equation}y^{(r)} = F(t; y, \dot{y},\cdots, y^{(r-1)})\label{eq:main}
\end{equation}  
such that if the initial values of $y$ and its derivatives are on a (higher) tangent bundle of $\cM$, then the solution of the equation is in $\cM$ ($\cM$ is an {\it invariant manifold}). If there is a method $\Psi(t,\cdots, h)$ on $D_{\cM}\subset\cE$ producing the next iteration point on $D_{\cM}$, (here, $\cdots$ includes current and past iteration data of $y$ and its derivatives), and $\Psi$ is of order $p$, that means $\Psi(t,\cdots, h)-y(t+h)$ is $O(h^{p+1})$, then since $y(t+h)\in\cM$
\begin{equation}|\Psi(t,\cdots, h) - \piperp(\Psi(t,\cdots, h) )|\leq |\Psi(t,\cdots, h)-y(t+h)  |.\label{eq:higherOrder}\end{equation}
Thus, $\lvert \Psi(t,\cdots, h) - \piperp(\Psi(t,\cdots, h) )\rvert$ is also $O(h^{p+1})$, hence
$\lvert \piperp(\Psi(t,\cdots, h))-y(t+h)\rvert$ is $O(h^{p+1})$ by the triangle inequality, and $\piperp(\Psi(t,\cdots, h))\in\cM$.

For an ANP-retraction (\cref{sec:tubular}) $\pi$, $\lvert\Psi(t,\cdots, h) - \pi(\Psi(t,\cdots, h) )\rvert$ is also $O(h^{p+1})$, and $\pi(\Psi(t,\cdots, h))$ is an $O(h^{p+1})$ method by the triangle inequality. Here, we do not need to know details about $\Psi$. The main requirement is that if $\cM$ is an invariant manifold of \cref{eq:main} then the projected method will have the same order of accuracy as the original method. This projection method \cite{HairerWarner,SoellnerFuhrer} is well-known in the context of constrained differential equations.
\begin{corollary}Let $\Gamma$ be a Christoffel function of the Levi-Civita connection of a metric $\sfg$ on $\cM$. Assume $\Gamma$ extends to a function from $D_{\cM}$ to $\Lin(\cE\otimes\cE, \cE)$, the space of $\cE$-valued bilinear form on $\cE$. Then an iterative method with order of accuracy $p$ applied to the equation
  \begin{equation}
    \frac{d}{dt}\begin{bmatrix}x\\v \end{bmatrix} = \begin{bmatrix}v\\-\Gamma(x; v, v)\end{bmatrix}
    \label{eq:geo}
  \end{equation}
 on $\cE^2$ ANP-retracted to $T\cM$ also has an order of accuracy $p$.\label{cor:ANPT}
\end{corollary}
For example, we get a method of order 4 by projecting the Runge-Kutta method $RK4$.

Given a (torsion-free) connection $\nabla$ on $\cM$, with Christoffel function $\Gamma$, a {\it tangent retraction} (\cref{sec:tubular}) with respect to $\nabla$ is of {\it second-order} if
\begin{equation}\frac{\nabla}{dt}\{\frac{d}{dt}(t\mapsto \fR(x, tv)\}\vert_{t=0} =\frac{d}{dt^2}\{t\mapsto\fR(x, tv)\}\vert_{t=0} + \Gamma(x, v, v) = 0.\label{eq:secondRetraction}
\end{equation}
\begin{theorem}1. Let $\pi$ be an $\cE$-tubular retraction of class $C^3$. Then its differential $\pi'$ at $x\in \cM$ maps $\cE$ to $T_x\cM\subset \cE$. Considered as a map from $\cE$ to $\cE$, $\pi'(x)$ is a linear projection from $\cE$ onto $T_x\cM$, $\pi'(x))^2=\pi'(x)$ for $x\in\cM$. In particular, $\piperp'(x) = \Pi^{\cE}(x)$, the $\cE$-orthogonal projection to $T_x\cM$.\hfill\break
  2. With $\pi$ above, assume $\nabla$ is a torsion-free connection with $C^1$-Christoffel function $\Gamma$. Then $\fR$ below is defined in a tubular neighborhood of $\cM\subset T\cM$ 
\begin{equation}
\fR(x, v) := \pi(x + v -\frac{1}{2}\pi'(x)\Gamma(x; v, v)).\label{eq:geoiter2}
\end{equation}
In that neighborhood, $\fR$ is a second-order tangent retraction for $\nabla$.\label{theo:geodesic}
\end{theorem}
\begin{proof} Since $\pi(x)\in\cM$, $\pi\circ\pi(x) = \pi(x)$ implies $\pi'(\pi(x))\pi'(x) = \pi'(x)$ for $x\in D_{\cM}$ by the chain rule. Evaluate at $x\in\cM$, this shows $\pi'(x)^2=\pi'(x)$. Differentiate $\pi(x)=x$ in direction $v\in T_x\cM$, we get $\pi'(x)v=v$.

With the assumed smoothness of $\pi$, for $|v|$ sufficiently small, $\fR$ in \cref{eq:geoiter2} is well-defined. For \cref{eq:secondRetraction}, let $y(t) = \fR(x, tv) = \pi(x + tv -\frac{t^2}{2}\pi'(x)\Gamma(x; v, v))$,
  $$\begin{gathered}\dot{y}(t) =  \pi'\{x + tv -\frac{t^2}{2}\pi'(x)\Gamma(x; v, v)\}(v-t\pi'(x)\Gamma(x; v, v)).
\end{gathered}$$
Thus, $\dot{y}(0) = \pi'(x)v = v$, or $\fR$ is a tangent retraction. Next,
$$\begin{gathered}\ddot{y}(t) = \pi^{(2)}\{x + tv -\frac{t^2}{2}\pi'(x)\Gamma(x; v, v);
v-t\pi'(x)\Gamma(x; v, v), v-t\pi'(x)\Gamma(x; v, v) \}\\
  -\pi'\{x + tv -\frac{t^2}{2}\pi'(x)\Gamma(x; v, v)\}(\pi'(x)\Gamma(x; v, v)),\\
  \Rightarrow \ddot{y}(0) = \pi^{(2)}(x ; v, v ) -\pi'(x)\Gamma(x; v, v)
\end{gathered}
$$
as $\pi'(x)^2=\pi'(x)$. Hence, the left-hand side of \cref{eq:secondRetraction} is
$$\left(\ddot{y}(t) +\Gamma(y(t); \dot{y}(t), \dot{y}(t))\right)\vert_{t=0} =
\pi^{(2)}(x ; v, v ) +(\dI_{\cE}-\pi'(x))\Gamma(x; v, v).$$
Note that since $\nabla$ is a connection, the left-hand side of the above is a tangent vector in $T_x\cM$. Since $\pi'(x)$ is the identity on $T_x\cM$, apply $\pi'(x)$ to both sides and use $\pi'(x)^2=\pi'(x)$, we have
$$\left(\ddot{y}(t) +\Gamma(y(t); \dot{y}(t), \dot{y}(t))\right)\vert_{t=0} =
\pi'(x)\pi^{(2)}(x ; v, v ).$$
It remains to show $\pi'(x)\pi^{(2)}(x ; v, v )=0$ (this is the Weingarten lemma applied to the linear projection $\pi'(x)$). Since $\pi'(y)^2v=\pi'(y)v$ holds {\it for all} $y\in \cM$, differentiate this equality in the {\it tangent direction} $v\in T_x\cM$
$$\pi^{(2)}(x; v, \pi'(x)v) + \pi'(x)\pi^{(2)}(x; v, v)= \pi^{(2)}(x; v, v).
$$
Since $\pi'(x)v=v$, canceling the first term with the right-hand side, we get $\pi'(x)\pi^{(2)}(x; v, v)=0$. This proves $\fR$ is a second-order retraction.
\end{proof}
In particular, for the embedded metric $\sfg = \dI_n$ and $\pi=\piperp$, then $\piperp'(x)=\Pi^{\cE}(x)$ and the Christoffel function of the Levi-Civita connection $\Gamma(x; v, v) = -(\Pi^{\cE})'(x; v)v$ is normal to $T_x\cM$ \cite{AbM}, thus, 
\cref{eq:geoiter2} reduces to $\fR(x, v) = \piperp(x + v)$. This retraction is known to be of second-order from \cite{AbM} and is used in \cite{SHSW}.

For a simple new example, consider the constrained hypersurface $\cM\subset\R^n$ defined by $c(x)-1=0$ for a scalar positively homogeneous function $c$ of order $\alpha\neq 0$, $c(tx) = t^{\alpha}c(x)$ for $t> 0$. Even for an ellipsoid ($\alpha=2$), $\piperp$ is in general complicated to solve. Consider the rescaling retraction $\pi(x)=\pi_s(x)=c(x)^{-\frac{1}{\alpha}}x$, defined in a neighborhood of $\cM$ where $c(x)>0$. For $x\in\cM$
\begin{equation}\pi'(x)\omega = -\frac{c'(x)\omega}{\alpha}c(x)^{-\frac{1}{\alpha}-1}x + c(x)^{-\frac{1}{\alpha}}\omega= \omega -\frac{c'(x)\omega}{\alpha}x,\quad\omega\in\R^n.\end{equation}
We can verify $\pi'(x)^2=\pi'(x)$. We can use \cref{eq:geoiter2} with $\pi_s$ for any metric on $\cM$, in particular, for $\sfg=\dI_n$ with $\Gamma(x; v, v) = \frac{c^{(2)}(x; v, v)}{|c'(x)|^2}c'(x)^{\sfT}$ \cite{Edelman_1999}.

For the group $\cM=\SL(N)$ of matrices of determinant $1$, a retraction from a neighborhood of matrices in $\R^N$ of positive determinant is the rescaled $\pi_{\det}:A\mapsto \det(A)^{-\frac{1}{N}}A$. We have for $A\in\SL(N)$, $\omega\in \cE=\R^{N\times N}$
$$\pi'(A)\omega = \lim_{t\to 0}\frac{d}{dt}\{\det(A+t\omega)^{-\frac{1}{N}}(A+t\omega)\} = \omega-\frac{1}{N}\Tr (A^{-1}\omega)A$$
by the Jacobi's determinant's formula and $\det(A) = 1$. Note, the projection $x(x^{-1}\omega)_{\pg}$, for $\pg(x^{-1}\omega) = x^{-1}\omega - \frac{\Tr(x^{-1}\omega)}{N}\dI_N$ in \cref{eq:Gproj} is exactly $\pi'(x)$.

For the Stiefel manifold in \cref{subsec:stiefel}, it is known \cite{AbM} $\piperp$ is given by the polar decomposition, thus, by \cref{eq:geoiter2}, a second-order retraction for the metric \cref{eq:stiefmetric} at $(Y,\xi)\in T\cM$ is given by the orthogonal component in the polar decomposition of
$$\begin{gathered}Y + \xi -\frac{1}{2}\Pi(Y)\left(Y(\xi^{\sfT}\xi )_{\sym} +2\frac{\alpha_0 - \alpha_1}{\alpha_0}K_0(\xi\xi^{\sfT})_{\sym}Y\right)\\
=Y + \xi -\frac{\alpha_0 - \alpha_1}{\alpha_0}(\xi-Y(Y^{\sfT}\xi))\xi^{\sfT}Y.
\end{gathered}$$
We can use these retractions to simulate a random walk approximation of a Riemannian Brownian motions \cite{SHSW} (where $\sigmam = \Pi\sigma$) where in \cref{alg:Simulation}, the iteration using a second-order retraction is given by
\begin{equation} X_{i,j+1}\gets \fR(X_{ij},  (\frac{h\dim{\cM}}{ \langle \sigmam(X_{ij})\Delta, \sigmam(X_{ij})\Delta\rangle_{\sfg(X_{ij})}})^{\frac{1}{2}}\sigmam(X_{ij})\Delta).
  \end{equation}
\subsection{The stochastic projection method}\label{subsec:projstochastic}
In \cite{ZhouZhangHongSong}, the authors studied the projection method for SDEs using the fundamental theorem of stochastic approximation in \cite{MilsteinTre}. The result is discussed in detail for one constraint. We will explain and extend it below.

Consider an SDE on $\cE$ with $\cM\subset\cE$ an invariant manifold, as in \cref{prop:conservedIto}, or in an equivalent Stratonovich form. Let $X_{t, x}(T)$ be the (exact) solution of this equation (in $\R^n$)with $X(t) = x\in\R^n$ at $T\geq t\in\R$. Consider an one-step approximation $\bar{X}_{t,x}(t+h) = x+A(t, x, h, W(\theta) - W(t))$ for $t\leq \theta\leq t+h$ on $\cE=\R^n$, where $A$ is continuous in $\R^n$. Assume (see \cite[theorem 1.1.1]{MilsteinTre}) $\bar{X}_{t,x}$ is of order of accuracy $p_1$ for the expectation of deviation and order of accuracy $p_2$ for the mean-square deviation, that is
\begin{gather}\lvert \E(X_{t, x}(t+h) - \bar{X}_{t, x}(t+h)) \rvert \leq K_1(1+|x|^2)^{\frac{1}{2}}h^{p_1}\label{eq:esp1}\\
\lvert [\E\lvert X_{t, x}(t+h) - \bar{X}_{t, x}(t+h) \rvert^2]^{\frac{1}{2}} \leq K_2(1+|x|^2)^{\frac{1}{2}}h^{p_2}  \label{eq:esp2}
\end{gather}
where $p_1\geq \frac{1}{2}, p_1\geq p_2+\frac{1}{2}$. Assume the global Lipschitz condition
\begin{equation}\lvert b(t, x) - b(t, y)\rvert + \sum_{i=1}^m \lvert \sigma_i(t, x) - \sigma_i(t, y)\rvert \leq K_0|x-y|\label{eq:lips}
\end{equation}
for $x, y\in \R^n$, $t\in[t_0, T]$ (the time interval under consideration), and $\sigma_i$'s are columns of $\sigma$, then the fundamental theorem of the mean-square order of convergence \cite[theorem 1.1.1]{MilsteinTre}, \cite[theorem 3.3]{ZhouZhangHongSong} shows for any $N$ and $k=0,1,\cdots , N$, there is a bound on the $k$-step approximation
\begin{equation}
\lvert [\E\lvert X_{t_0, X_0}(t_k) - \bar{X}_{t_0, X_0}(t_k) \rvert^2]^{\frac{1}{2}} \leq K(1+\E|X_0|^2)^{\frac{1}{2}}h^{p_2-\frac{1}{2}}.  
\end{equation}
The approximation $\bar{X}_{t,x}$ is in $\R^n$. Similar to the ODE case, but under further assumptions, the projection method could be modified to a method on $\R^n$ to an approximation in $\cM$ with the same expectation and mean-square accuracy. Thus, the $k$-steps bound still applies for the projected method on $\cM$. We need the following result on $\R^n$ (theorem 3.4 in \cite{ZhouZhangHongSong}, which cited \cite{MilsteinTre}. As before, the proof follows from the triangle inequality.)
\begin{proposition}Let the one-step approximation $\bar{X}_{t, x}$ satisfy the condition of the fundamental theorem \cite[theorem 1.1.1]{MilsteinTre}. Suppose that another approximation $\hat{X}_{t,x}$ is such that
\begin{gather}\lvert \E(\hat{X}_{t, x}(t+h) - \bar{X}_{t, x}(t+h)) \rvert \leq K_3(1+|x|^2)^{\frac{1}{2}}h^{p_1}\\
 [\E\lvert \hat{X}_{t, x}(t+h) - \bar{X}_{t, x}(t+h) \rvert^2]^{\frac{1}{2}} \leq K_4(1+|x|^2)^{\frac{1}{2}}h^{p_2}  
\end{gather}
Then the method based on $\hat{X}_{t, x}$ will also have the order of accuracy $p_1$ and $p_2$ for expectation and mean-square deviation. \label{prop:triangl}
\end{proposition}  
\begin{proof} By using the triangle inequality, for example
  $$\begin{gathered}  \lvert \E( X_{t, x}(t+h)-\hat{X}_{t, x}(t+h) ) \rvert \leq
    \lvert \E(X_{t, x}(t+h)- \bar{X}_{t, x}(t+h)  ) \rvert \\
    + \lvert \E( \bar{X}_{t, x}(t+h)-\hat{X}_{t, x}(t+h)) \rvert
  \end{gathered}$$
then apply the estimates.  
\end{proof}
The argument that an ANP-retraction preserves the order of accuracy in the ODE case needs a modification in the SDE case, since the Gaussian distribution is unbounded, so even when $h$ is small, there is a small chance the stochastic move lies far from the tubular neighborhood where the ANP-retraction is defined. Following \cite{MilsteinTreSymp,ZhouZhangHongSong}, we truncate the stochastic moves. For a scalar Wiener process $w_t$, represent $\Delta w_t(h) =w_{t+h}-w_t= \xi h^{\frac{1}{2}}$ where $\xi\sim N(0, 1)$ is standard normal. Fix $r \geq 1$, let $A_h = (2r\lvert \ln(h)\rvert)^{\frac{1}{2}}$ and
\begin{equation}\zeta_h := \begin{cases}\xi, \lvert \xi \rvert \leq A_h,\\
    A_h, \xi > A_h,\\
    -A_h,\xi < - A_h.
  \end{cases}\label{eq:STcutoff}
\end{equation}
Then $\zeta_h$ is {\it truncated} from $\xi$. For a multi-dimensional Wiener process $W_t = (W_{1,t},\cdots, W_{m, t})$, represent $\Delta W_{i, t}(h) = \sqrt{h} \xi_i$, where $\xi_i\sim N(0, 1)$ and $\xi_i$'s are independent, then we can truncate $\xi_i$ to $\zeta_{i, h}$ in the method $\bar{X}$. 
\begin{theorem}\label{theo:proj}Assume the global Lipschitz condition in \cref{eq:lips}, let $\bar{X}$ be an approximated method on $\R^n$ for a SDE where $\cM$ is an invariant manifold, and $\tilde{X}$ be the method obtained by replacing the stochastic step $\Delta w(h)=\xi\sqrt{h}$ with $\zeta_h\sqrt{h}$ in \cref{eq:STcutoff}. Assume the truncated method $\tilde{X}$ is with orders of accuracy $p_1, p_2 \geq p_1+\frac{1}{2}$ for expectation and mean-square. Let $\pi$ be an ANP-retraction defined in a tubular neighborhood $\cU$ of $\cM\subset\cE$. Assume there is $h_0>0$ such that for $h\leq h_0$, both $\hat{X}=\piperp(\tilde{X})$ (realizing the smallest distance between $\tilde{X}$ and $\cM$) and $\pi(\tilde{X})$ are defined. Then the method $\check{X}_{t, h} = \pi(\tilde{X}_{t, h})$ also has orders of accuracy $p_1, p_2 \geq p_1+\frac{1}{2}$ for expectation and mean-square.
\end{theorem}
\begin{proof} Replacing $\Delta W(h) = W(\theta)-W(t)$ by $\hat{\Delta} W(h) = \zeta_h h^{\frac{1}{2}}$ then $\hat{\Delta} W(h)$ is bounded and goes to zero when $h$ is sufficiently small, hence $A(t, x, \hat{\Delta} W(h))$ is close to zero, thus, for small enough $h$, $\hat{X} = \piperp(\tilde{X})$ and $\check{X} = \pi(\tilde{X})$ exist. Since $\hat{X}_{t, x}(t+h)$ is nearest to $\tilde{X}_{t, x}(t+h)$ on $\cM$
$$\lvert \hat{X}_{t, x}(t+h)- \tilde{X}_{t, h}(t+h)\rvert \leq
\lvert X_{t, x}(t+h)- \tilde{X}_{t, h}(t+h)\rvert.
$$
From here, we can take expectation. The estimates for the right-hand side using \cref{eq:esp1,eq:esp2} imply the same estimates on the left, then we can use \cref{prop:triangl} to proves the statement for $\hat{X}$. The estimate $|\tilde{X}-\pi(\tilde{X}) | \leq C|\tilde{X}-\piperp{\tilde{X}}|$ implies the condition in \cref{prop:triangl} is satisfied for $\check{X}$.
\end{proof}
Thus, if we use the Euler-Maruyama's method \cite{KUHNEL2019411} for an It{\=o} equation, in \cref{alg:Simulation}, $X_{i,j+1}$ is computed from $X_{ij}$ by
\begin{equation} X_{i,j+1}\gets \pi(X_{ij} + h^{\frac{1}{2}}\sigmam(X_{ij}, jh)\Delta + h\mu(X_{ij}, jh)).\end{equation}
If we use the Euler-Heun's method for a Stratonovich equation \cite{KUHNEL2019411} then
  \begin{equation} X_{i,j+1}\gets \pi(X_{ij}  + h\mu_S (X_{ij}, jh)  + \frac{h^{\frac{1}{2}}}{2}(\sigmam(X_{ij}, jh) + \sigmam(X_{ij}+h^{\frac{1}{2}}\sigmam(X_{ij}, jh)\Delta, jh))\Delta).
  \end{equation}

\subsection{Retractive Euler-Maruyama methods}\label{sec:retractivesim}We now show that given a tangent retraction $\fR$ on $\cM\subset\cE$, we can construct an integrator for the SDE in \cref{prop:conservedIto}. We do not assume a relationship between the retraction and the SDE. Consider $\fR(x, .):\xi\mapsto \fR(x, \xi)$ as a function of $\xi\in T_x\cM\subset\cE$, mapping to $\cM\subset\cE$. Let $\fR^{(2)}(x, \xi;., .)$ be the Hessian (in $\xi$) of $\fR$, which evaluates on vectors $v, w\in T_x\cM$ to a vector $\fR^{(2)}(x, \xi; v, w)\in\cE$. If $\fR$ is of class $C^3$, assume for small $h\in\R$
\begin{equation}\fR(x, hv) = x + hv + \frac{h^2}{2}\fR^{(2)}(x,0; v, v) + O(h^3(1+|x|^2)^{\frac{1}{2}}).
\end{equation}
Note that $\fR^{(2)}$ is bilinear in $(T_x\cM)^2$, we will extend it bilinearly to $\cE^2$. For example, consider the unit sphere $x^{\sfT}x = 1$ with the rescaling retraction $\fR(x, v) = \frac{1}{|x+v|}(x+v)$. Then $\fR^{(2)}(x, 0; v, v)= -xv^{\sfT}v$ and
\begin{equation} \fR(x, hv) = x + hv - \frac{h^2}{2}xv^{\sfT} v + O(h^3(1+|x|^2)^{\frac{1}{2}}).\label{eq:dominate}
\end{equation}
by Taylor expansion, since $x$ is bounded on the sphere.
\begin{theorem}\label{theo:REM}Let $W_t$ be a Wiener process on an inner product space $\cE_W=\R^k$, with an orthonormal basis $\{w_j| j=1,\cdots, k\}$, $\sigmam(X_t, t)$ is a map from $\cE_W$ to $\cE$, such that $\sigmam(X_t, t)\cE_W\subset T_{X_t}\cM$, and $(\ttA, \ttM) = (2\mu, \sigmam\sigmam^{\sfT})$ is a SOO on $\cM$. Consider the SDE on $\cM\in\cE$, represented by the It{\=o} equation on $\cE$
\begin{equation} dX_t =\mu(X_t, t)dt + \sigmam(X_t, t)dW_t.\label{eq:ItoSDE}\end{equation}
In that case, $\mu_{\fR}$ defined below is in $T_{X_t}\cM$
\begin{equation}\mu_{\fR}(X_t, t) := \mu(X_t, t) -\frac{1}{2}\sum_{j=1}^k\fR^{(2)}(X_t, 0;\sigmam(X_t, t)w_j, \sigmam(X_t, t)w_j).
\label{eq:muradj}
\end{equation}
Given $X_0\in \cM$, set $\Delta_{W_{t}}(h):= W_{t+h}-W_t$. Assume $\mu(x, t)$ and $\sigmam$ has partial derivatives with respect to $t$ that grow at most linearly in $x$ as $\lvert x \rvert \to \infty$ and the global Lipschitz condition as in \cite[theorem 1.1.1]{MilsteinTre} 
\begin{equation}\lvert \mu(x, t) -\mu(y, t) \rvert + \sum_{j=1}^k\lvert \sigmam(x, t)w_j-\sigmam(y, t)w_j\rvert \leq K|x-y|.
\end{equation}
Assume that entries of $\fR^{(2)}$ is dominated by a constant times $(1+|x|^2)^{\frac{1}{2}}$, and entries of the representation of the linear map $\Delta\mapsto \fR^2(x, 0; \sigmam \Delta , \mu_{\fR})$ as a matrix is dominated similarly, and \cref{eq:dominate} is satisfied. then in the notation of that theorem, the \emph{retractive Euler-Maruyama method}
\begin{equation}\bar{X}(t+h) = \fR(X_t, \sigmam(X_t, t)\Delta_{W_{t}(h)} + h\mu_{\fR}(X_t, t))
\label{eq:retracEM}
\end{equation}
converges to a solution of \cref{eq:ItoSDE}, is of order of accuracy $p_1=\frac{3}{2}$ for expectation, $p_2 = 1$ for mean square deviation, and order of accuracy $p=\frac{1}{2}=p_2-\frac{1}{2}$.
\end{theorem}
Thus, using \cref{eq:retracEM}, in \cref{alg:Simulation}, $X_{i,j+1}$ is computed from $X_{ij}$ as
\begin{equation} X_{i,j+1}\gets \fR(X_{ij}, h^{\frac{1}{2}}\sigmam(X_{ij}, jh)\Delta + h\mu_{\fR}(X_{ij}, jh)).\end{equation}
In particular, if $X_t$ is a Riemannian Brownian motion and $\fR(X, v)$ is a second-order tangent retraction of the Levi-Civita connection with Christoffel function $\Gamma$, then $\fR^{(2)}(X_t,0; v, v)=-\Gamma(X_t, v, v)$, thus $\mu_{\fR} = 0$.
\begin{proof}Assume near $x\in\cM$, $\cM$ is defined by equations $\rC^i(x)=0, i=1\cdots n-\dim\cM$, then by the SOO assumption $(\rC^i)'\mu + \frac{1}{2}\Tr\{(\rC^i)^{(2)}\sigmam\sigmam^{\sfT}\}=0$, or
\begin{equation}(\rC^i)'\mu + \frac{1}{2} \sum_jw_j^{\sfT}\sigmam^{\sfT} (\rC^i)^{(2)}\sigmam w_j
=(\rC^i)'\mu + \frac{1}{2} \sum_j(\rC^i)^{(2)}(x, \sigmam w_j, \sigmam w_j)
=0\label{eq:proofEM1}
\end{equation}
where everything is evaluated at $(x, t)$. For $\xi\in T_x\cM$, differentiate the equations $\rC^i(\fR(x, h\xi)=0$ in $h$ twice then set $h=0$. Here, $\fR^{(1)}(x, .)(h\xi)$ denote the derivative in the $T_x\cM$ variable, evaluated at $(h\xi)$
$$\begin{gathered}(\rC^i)'(\fR(x, h\xi)) \fR^{(1)}(x, .)(h\xi)\xi =0,\\
(\rC^i)^{(2)}(x; \xi, \xi) + (\rC^i)'(x) \fR^{(2)}(x, 0; \xi, \xi) =0.
\end{gathered}
$$
Now, set $\xi = \sigmam w_j$ in the second equation, and subtract half the sum in $w_j$ with \cref{eq:proofEM1}, we get $(\rC^i)'\mu_{\fR}=0$, hence, $\mu_{\fR}$ is tangent to $\cM$.

Since $\mu_r\in T_{X_t\cM}$, hence $\sigmam \Delta_{W_t}(h) + h\mu_r\in T_{X_t}\cM$, the retraction in \cref{eq:retracEM} is well-defined for small $h$. We compare its Taylor expansion with the usual Euler-Maruyama (EM)'s method $\Bar{\Bar{X}}(t+h) = X_t + \sigmam(X_t, t)\Delta_{W_t(h)} + h\mu(X_t, t)$. Dropping some variables for brevity, the difference is
$$\begin{gathered}\bar{X}(t+h) -
\Bar{\Bar{X}}(t+h)  = 
X_t + \sigmam\Delta_{W_{t}(h)} + h\mu_{\fR}
+ \frac{1}{2}\fR^{(2)}(X_t, 0; \sigmam \Delta_{W_t}(h), \sigmam \Delta_{W_t}(h)) \\
+h\fR^{(2)}(X_t, 0 ; \sigmam \Delta_{W_t}(h), \mu_{\fR})
+ \frac{h^2}{2}\fR^{(2)}(X_t, 0; \mu_{\fR}, \mu_{\fR})\\
- (X_t + \sigmam\Delta_{W_t(h)} + h\mu)
+ O(h^3(1+|X_t|^2)^{\frac{1}{2}})\\
= \frac{h}{2}(-\sum_j \fR^{(2)}(X_t,0; \sigmam w_j, \sigmam w_j)
+ \frac{1}{2}\fR^{(2)}(X_t, 0; h^{-\frac{1}{2}}\sigmam \Delta_{W_t}(h), h^{-\frac{1}{2}}\sigmam \Delta_{W_t}(h))) \\
+h^{\frac{3}{2}}\fR^{(2)}(X_t, 0; h^{-\frac{1}{2}}\sigmam \Delta_{W_t}(h), \mu_{\fR})
+ \frac{h^2}{2}\fR^{(2)}(X_t, 0; \mu_{\fR}, \mu_{\fR})
+ O(h^3(1+|X_t|^2)^{\frac{1}{2}}).
\end{gathered}
$$
Applying \cref{lem:expect1} to the quadratic form $\Delta \mapsto \frac{1}{2}\fR^{(2)}(X_t, 0; \sigmam \Delta, \sigmam \Delta) $ to cancel the first two terms (in expectation) and \cref{lem:expect2} for each tensor contraction of (entries of) $\fR^{(2)}(X,0)$ with $\sigmam^{\sfT}\mu_{\fR}$, we have $|\E(\bar{X}(t+h) - \Bar{\Bar{X}}(t+h) )|=O(h^{\frac{3}{2}}(1+|x|^2)^{\frac{1}{2}})$, together with $p_1=\frac{3}{2}$ of the original Euler's method \cite[section 1.1.5]{MilsteinTre}, giving us the combined $p_1=\frac{3}{2}$ for the retractive EM method. Similarly, by \cref{lem:expect3}, with $\Delta = h^{-\frac{1}{2}}\Delta_{W_t}(h)$
$$\E\lvert \bar{X}(t+h) - \Bar{\Bar{X}}(t+h) |^2=\frac{h^2}{4}\E
(\fR^{(2)}(X_t, 0; \sigmam \Delta, \sigmam \Delta)
-\fR^{(2)}(X_t, 0; \sigmam w_j, \sigmam w_j))^2
+ O(K_{h,x})
$$
where $K_{h, x}=h^{\frac{5}{2}}(1+|x|^2)^{\frac{1}{2}}$, then apply \cref{lem:expect3}, we have $p_2=1$.
\end{proof}
For an autonomous system ($\mu$ and $\sigmam$ are not time-dependent), in the setting of \cite[section 2(d)]{AmstrongBrigoIntrinsic}, the map $\gamma_{X}:\cE_W\to \R^n$ below 
\begin{equation}\gamma_{_X}\Delta_W = \fR(X, \sigmam(X)\Delta_W + \mu_{\fR}(X)\delta t)
\end{equation}
where $\Delta_W \sim N(0, (\delta t)^{\frac{1}{2}}\dI_{\cE_W})$ corresponds to a 2-jet, which approximates the It{\=o} equation $dX_t = \mu(X_t) dt + \sigmam(X_t)dW_t$. Thus, through \cref{theo:REM}, we link tangent retractions with jets, and construct a jet to solve this SDE.

We briefly discuss here the random-walk-type integrator in \cite{SHSW}. Since for $X_t\in\cM$, $\sigmam(X_t)$ is a linear map and $\Delta_W(h)$ is a multivariate normal random variable, $\sigmam(X_t) \Delta_W(h)$ is \emph{multivariate normal} in $T_{X_t}\cM$, with covariant matrix $\sigmam\sigmam^{\sfT}$ (we dropped $X_t$ for brevity), which we assume to be invertible when restricted to $T_{X_t}\cM$. This defines an inner product $\sfg_{\sigmam}$ defined by $((\sigmam\sigmam^{\sfT})_{T_X\cM})^{-1}$ on $T_X\cM$. It is known \cite{chikuse_book,SHSW}, the rescaling (excluding the zero point) of a \emph{multivariate normal} variable on an inner product space produces a \emph{uniform} distribution on the unit sphere, hence, $\frac{1}{ | \sigmam \Delta_W(h) |^{\frac{1}{2}}_{\sfg_{\sigmam}}}\sigmam\Delta_W(h)$ 
($|.|_{\sfg_{\sigmam}}$ is a the norm under $\sfg_{\sigmam}$) is a \emph{uniform distribution} on the unit sphere of $T_X\cM$ under $\sfg_{\sigmam}$. Thus, this provides a way to sample that unit sphere. We can use this to approximate the Brownian motion as in \cite{SHSW}, if we have a second-order tangent retraction corresponding to the Levi-Civita connection for $\sfg_{\sigmam}$. If we do not use a second-order retraction, or if in the general equation (\ref{eq:ItoSDE}) the term $\mu_{\fR}$ is not zero then we will need to make an appropriate modification, the retraction $\fR(X_t, \frac{\sqrt{h\dim\cM}}{\lvert \sigmam \Delta_W(1)\rvert_{\sfg_{\sigmam}}}\sigmam\Delta_W(1) + h\mu_{\fR})$ is a candidate when comparing with \cref{eq:retracEM}. This will require further justification.

As an example, we consider a compact hypersurface defined by $\rC(x) = \sum_{i=1}^n d_ix_i^p=1$ (this class of hypersurfaces includes the ellipsoids) although the result here extends straight-forwardly to general homogeneous constraints as in \cref{sec:projgeo}.  Instead of using the second-order tangent retraction described there for the rescaling $\cE$-tubular retraction $\pi(q) = \frac{1}{\rC(q)^{\frac{1}{p}}}q$, we use the simple tangent retraction $\fR_{_+}(x, v) = \pi(x+v)$. In this case, we expand
\begin{equation}
\fR_{_+}(x, hv) = x+hv + \frac{h^2(1-p)\sum d_i x_i^{p-2} v_i^2}{2}x +O(h^3)
\end{equation}
using compactness to bound $(1+|x|^2)^{\frac{1}{2}}$. Using the second order term as an adjustment in \cref{eq:muradj}, in a \href{https://github.com/dnguyend/jax-rb/blob/main/tests/notebooks/TestRetractiveIntegrator.ipynb}{notebook} in \cite{LaplaceNumerical}, we verify that this method produces similar results to other simulation methods considered in this work.
\section{Experiments}
\subsection{Software Library}We develop a software library, JAX-RB \cite{LaplaceNumerical} implementing the basic geometry and simulation framework developed here. The package is organized into two main modules, \emph{manifold}, containing the geometric operations of the manifolds in \cref{sec:examples}, and \emph{simulation}, containing the different numerical simulation methods in \cref{sec:numeric}. The package contains workbooks and scripts illustrating the basic function calls and performing numerical tests. The derivatives are computed using the Python package JAX \cite{jax2018github}. We also take advantage of JAX's \emph{vmap} mechanism and just-in-time compilation to efficiently vectorize the code.

All results in this section are obtained using the package. For the sphere and the manifolds in \cref{sec:examples}, we perform several sanity tests, including confirming the metric compatibility of the Levi-Civita connection, and comparing the Laplacian computed using orthonormal vector fields near a point $x\in\cM$, and computed using \cref{theo:laplace,theo:lift}. We also compare simulation results from different methods, and compare long term simulations with uniform (with respect to the Riemannian metric) sampling on the manifold for the sphere, SO(N), Stiefel and Grassmann manifolds.

\subsection{Simulation versus heat kernel integration on the sphere}
We first compare simulation results with integration using the heat kernel for the two and three-dimensional spheres. For the $2$-sphere, where the heat kernel is given by a theta function
\cite{HansES,mpmath}, the heat kernel method gives the expectation at $T=2$ for cost $\phi^{\frac{5}{2}}$, $\phi =\cos^{-1}(\frac{x_1}{r})$ for a point $(x_1,\cdots, x_n)\in \Sph^2$, diffusion coefficient $.4$, radius $r=3$ of $.299$, while simulations at $1000$ paths and $1000$ subdivision yield $.282$ to $.288$ (geodesic, It{\=o} and Stratonovich). For $d=3$ with cost $\phi^{\frac{3}{2}}+\phi^{\frac{5}{2}}$, the heat kernel result is $1.02$, while the simulations results are between $1.05$ to $1.066$, as can be seen in the notebook \href{https://github.com/dnguyend/jax-rb/blob/main/tests/notebooks/test_heat_kernel.ipynb}{test\_heat\_kernel.ipynb} of \cite{LaplaceNumerical}.

\subsection{Comparing Brownian simulations using different numerical methods}\label{sec:compare}
We consider the three main simulation methods, with some variances. 
\begin{itemize}\item It{\=o} equation with projection. Using \cref{eq:BrownianIto} together with \cref{theo:proj}. The main integration method is Euler-Maruyama with step length $h$.
\item Stratonovich equation with projection. Using \cref{eq:BrownianStrat} or \cref{eq:BrownianStratAlt} together with \cref{theo:proj}. The main integration method is Euler-Heun.
\item Geodesic\slash retractive simulation. This method is based on \cref{sec:retractivesim}, and comes in two flavors, the retraction in \cref{eq:retracEM}, and the normalized retraction discussed at the end of that section. We use a second-order tangent retraction of the Levi-Civita connection as in \cref{theo:geodesic}, thus, the adjusted drift $\mu_{\fR}$ is zero, and the normalized method reduces to the framework in \cite{Gangolli,Jorgensen,SHSW}. We test both flavors. If we have an orthonormal basis of the tangent space at $x$, a sample vector from the uniform distribution on the unit sphere in $\R^d$, $d=\dim\cM$ could be used to generate a tangent vector of length $d^{\frac{1}{2}}h$. We apply this observation for left-invariant groups, where we use the left translation of an orthonormal basis of the Lie algebra.
\end{itemize}
Methods involving normalizing the moves to a fixed length as in the geodesic random walk could help avoid the cutoff step. Higher-order methods could also be used, we will leave that to future research.

We use the rescaling retraction for the sphere and for $\SL(N)$, the addition retraction for $\GL^+(N)$, $\Sd{N}$ and $\Aff(N)$, the polar decomposition for $\SO(N)$, Stiefel, Grassmann and $\SE(N)$ (for the rotation component), with the second-order adjustment by the Christoffel function, see \cite{LaplaceNumerical}.

As explained in \cref{alg:Simulation}, we compute the expectation $\E(\int_0^Tg(X_s, s)ds + f(X_{T}, T), m_{RB})$ where $m_{RB}$ is the Riemannian-Brownian measure on paths on $\cM$.
The expectation is evaluated on $n_{path}=1000$ paths with a ladder of $n_{div}=200,500,700$ subdivisions. For the noncompact manifold $\Sd{n}$, we use a bounded function when the path goes to infinity. For the left-invariant metrics, we bound the eigenvalue of $\cI$ so that the Brownian motion does not grow too fast (This is important for noncompact groups such as $\GL(N), \SL(N), \Aff(N)$, and $\SE(N)$.)

We include the Grassmann manifold in our study. While some analysis may be required to justify the projection method and the tangent retraction on quotient manifolds, the simulation results look reasonable, converge to the uniform distribution at large $T$ when applied to lifted functions from the submersion.

For the tests we conducted, all three methods and their variances produce results matching within a reasonable tolerance, see \cref{tbl:tblBySubdir,tbl:tblBySimType}. In particular, for \cref{tbl:tblBySimType}, we show the ranges of expected values produced by different runs and methods, which confirms good agreements. 

In the integral $\E(\int_0^Tg(X_s, s)ds + f(X_{T}, T), m_{RB})$, we consider two cases, $g=0$ and $g\neq 0$. For example, for the manifold of positive definite matrices, the first group (final value only) corresponds to $g=0$ and  $f(X_T) = \lvert (X_{T})_{11}\rvert$ ($T$ is the final time). The second group is with the same $f$ with $g = \max((X_t)_{11}), 0)$. The code for all experiments is available in \cite{LaplaceNumerical}. The large discrepancies for $\Sd{3}$ at $T=10$ shows long time simulation for a non compact manifold requires further refinements.\begin{table}
  \resizebox{\textwidth}{!}{
\begin{tabular}{llrrrrrrrr}
\toprule
        & Cost type & \multicolumn{4}{l}{At final time only} & \multicolumn{4}{l}{With intermediate cost} \\
        & Final time &               0.5  &   2.0  &   3.0  &   10.0 &                   0.5  &    2.0  &    3.0  &       10.0 \\
\midrule
$\Aff(3)$ & Geodesic &              0.718 &  0.738 &  0.737 &  0.632 &                  0.777 &   1.948 &   3.812 &    116.227 \\
        & It{\=o} &              0.718 &  0.738 &  0.737 &  0.634 &                  0.777 &   1.948 &   3.805 &    114.436 \\
        & Stratonovich &              0.718 &  0.738 &  0.737 &  0.636 &                  0.777 &   1.945 &   3.796 &    113.401 \\
$\GL^+(2)$ & Geodesic &              0.705 &  0.702 &  0.699 &  0.694 &                  0.769 &   1.784 &   3.275 &     39.407 \\
        & It{\=o} &              0.705 &  0.702 &  0.699 &  0.693 &                  0.769 &   1.786 &   3.280 &     39.498 \\
        & Stratonovich &              0.705 &  0.702 &  0.699 &  0.693 &                  0.769 &   1.785 &   3.280 &     39.462 \\
$\SE(3)$ & Geodesic &              0.947 &  0.801 &  0.734 &  0.454 &                  1.005 &   1.670 &   2.494 &     12.839 \\
        & It{\=o} &              0.947 &  0.802 &  0.735 &  0.454 &                  1.005 &   1.668 &   2.495 &     12.853 \\
        & Stratonovich &              0.947 &  0.802 &  0.735 &  0.454 &                  1.005 &   1.668 &   2.494 &     12.835 \\
$\SL(3)$ & Geodesic &              0.237 &  0.245 &  0.227 &  0.079 &                  0.303 &   1.445 &   3.242 &     96.066 \\
        & It{\=o} &              0.238 &  0.243 &  0.226 &  0.081 &                  0.303 &   1.444 &   3.239 &     82.257 \\
        & Stratonovich &              0.238 &  0.243 &  0.226 &  0.079 &                  0.303 &   1.444 &   3.236 &     95.928 \\
$\SO(3)$ & Geodesic &              0.931 &  0.769 &  0.681 &  0.408 &                  0.990 &   1.581 &   2.350 &     11.410 \\
        & It{\=o} &              0.932 &  0.769 &  0.681 &  0.410 &                  0.990 &   1.581 &   2.348 &     11.407 \\
        & Stratonovich &              0.932 &  0.769 &  0.681 &  0.409 &                  0.990 &   1.581 &   2.348 &     11.387 \\
$\Gr{5}{3}$ & Geodesic &              0.631 &  0.601 &  0.599 &  0.597 &                  0.723 &   1.820 &   3.320 &     30.599 \\
        & It{\=o} &              0.632 &  0.603 &  0.600 &  0.604 &                  0.721 &   1.820 &   3.319 &     30.624 \\
        & Stratonovich &              0.631 &  0.601 &  0.598 &  0.597 &                  0.721 &   1.818 &   3.319 &     30.612 \\
$\Sd{3}$ & Geodesic &              1.115 &  1.442 &  1.843 &  9.750 &                  1.294 &  10.083 &  42.610 &  37710.684 \\
        & It{\=o} &              1.115 &  1.445 &  1.850 &  9.456 &                  1.296 &  10.130 &  42.543 &  26796.579 \\
        & Stratonovich &              1.115 &  1.444 &  1.850 &  9.532 &                  1.295 &  10.126 &  42.787 &  28288.189 \\
$\Sph^9$ & Geodesic &              0.299 &  1.364 &  1.904 &  3.272 &                  0.349 &   1.637 &   2.315 &      4.125 \\
        & It{\=o} &              0.300 &  1.389 &  1.946 &  3.316 &                  0.352 &   1.656 &   2.343 &      4.060 \\
        & Stratonovich &              0.298 &  1.365 &  1.903 &  3.283 &                  0.350 &   1.638 &   2.315 &      4.123 \\
$\St{5}{3}$ & Geodesic &              0.445 &  0.373 &  0.379 &  0.382 &                  0.518 &   0.884 &   1.352 &      9.745 \\
        & It{\=o} &              0.447 &  0.374 &  0.379 &  0.382 &                  0.518 &   0.883 &   1.346 &      9.754 \\
        & Stratonovich &              0.447 &  0.375 &  0.381 &  0.380 &                  0.517 &   0.885 &   1.343 &      9.741 \\
\bottomrule
\end{tabular}
  }
\caption{Riemannian Brownian simulations by different simulation methods. Simulating $1000$ paths, each with $700$ subdivisions for two cost types and four final times. Displayed by groups of simulation methods.}
\label{tbl:tblBySubdir}
\end{table}

\begin{table}
  \resizebox{\textwidth}{!}{
    \begin{tabular}{llllllllll}
\toprule
        & Cost type & \multicolumn{4}{l}{At final time only} & \multicolumn{4}{l}{With intermediate cost} \\
        & Final time &               0.5  &          2.0  &          3.0  &          10.0 &                   0.5  &           2.0  &           3.0  &                  10.0 \\
manifold & \# div &                    &               &               &               &                        &                &                &                       \\
\midrule
$\Aff(3)$ & 200 &       0.72 (0.002) &  0.74 (0.013) &  0.74 (0.016) &  0.65 (0.026) &           0.78 (0.002) &   1.88 (0.066) &   3.79 (0.603) &       105.74 (55.144) \\
        & 500 &       0.72 (0.002) &  0.74 (0.006) &  0.74 (0.013) &  0.64 (0.028) &           0.78 (0.002) &   1.96 (0.092) &   3.71 (0.209) &       122.37 (19.071) \\
        & 700 &       0.72 (0.002) &   0.74 (0.01) &  0.74 (0.006) &   0.63 (0.03) &           0.78 (0.005) &   1.95 (0.133) &   3.81 (0.366) &       115.07 (22.418) \\
$\GL^+(2)$ & 200 &       0.71 (0.002) &   0.7 (0.003) &   0.7 (0.006) &   0.7 (0.005) &           0.77 (0.001) &   1.78 (0.017) &   3.26 (0.052) &         38.01 (1.309) \\
        & 500 &       0.71 (0.002) &   0.7 (0.002) &   0.7 (0.003) &  0.69 (0.005) &           0.77 (0.001) &   1.78 (0.007) &   3.26 (0.078) &         39.63 (2.974) \\
        & 700 &        0.7 (0.002) &   0.7 (0.002) &   0.7 (0.005) &  0.69 (0.008) &           0.77 (0.001) &    1.78 (0.05) &   3.28 (0.056) &         39.44 (1.198) \\
$\SE(3)$ & 200 &       0.95 (0.002) &  0.81 (0.005) &  0.74 (0.025) &  0.45 (0.014) &            1.0 (0.004) &   1.65 (0.017) &    2.5 (0.059) &          13.0 (0.686) \\
        & 500 &       0.94 (0.002) &   0.8 (0.011) &   0.73 (0.02) &  0.45 (0.009) &            1.0 (0.002) &   1.66 (0.014) &   2.52 (0.042) &           12.86 (0.3) \\
        & 700 &       0.95 (0.002) &   0.8 (0.006) &  0.73 (0.012) &  0.45 (0.022) &           1.01 (0.003) &   1.67 (0.023) &   2.49 (0.031) &         12.84 (0.689) \\
$\SL(3)$ & 200 &       0.24 (0.005) &  0.25 (0.019) &  0.22 (0.012) &  0.08 (0.015) &            0.3 (0.003) &   1.45 (0.049) &   3.26 (0.279) &        97.66 (22.368) \\
        & 500 &       0.24 (0.007) &  0.24 (0.012) &   0.22 (0.01) &  0.08 (0.016) &            0.3 (0.007) &    1.4 (0.051) &   3.28 (0.085) &        92.79 (16.105) \\
        & 700 &       0.24 (0.004) &  0.24 (0.008) &  0.23 (0.011) &  0.08 (0.012) &            0.3 (0.001) &   1.44 (0.032) &   3.24 (0.231) &        92.58 (20.166) \\
$\SO(3)$ & 200 &       0.93 (0.004) &  0.77 (0.026) &  0.68 (0.021) &  0.41 (0.038) &           0.99 (0.004) &   1.58 (0.024) &   2.35 (0.038) &          11.38 (0.94) \\
        & 500 &       0.93 (0.007) &  0.76 (0.017) &  0.68 (0.023) &  0.41 (0.032) &           0.99 (0.004) &   1.58 (0.025) &   2.34 (0.042) &         11.49 (0.345) \\
        & 700 &       0.93 (0.009) &  0.77 (0.022) &  0.68 (0.019) &   0.41 (0.03) &           0.99 (0.006) &   1.58 (0.012) &   2.35 (0.035) &          11.4 (0.556) \\
$\Gr{5}{3}$ & 200 &       0.63 (0.024) &   0.6 (0.024) &   0.6 (0.016) &   0.6 (0.033) &           0.73 (0.011) &   1.82 (0.026) &   3.31 (0.036) &         30.48 (0.354) \\
        & 500 &       0.64 (0.018) &   0.61 (0.02) &   0.6 (0.024) &   0.6 (0.024) &           0.71 (0.014) &   1.81 (0.048) &    3.32 (0.05) &         30.54 (0.266) \\
        & 700 &       0.63 (0.007) &    0.6 (0.02) &   0.6 (0.019) &   0.6 (0.029) &           0.72 (0.016) &    1.82 (0.01) &   3.32 (0.067) &         30.61 (0.248) \\
$\Sd{3}$ & 200 &       1.12 (0.007) &  1.45 (0.031) &  1.79 (0.119) &  9.51 (3.494) &             1.3 (0.03) &   9.71 (0.918) &   39.66 (4.71) &  30000.49 (24816.615) \\
        & 500 &       1.12 (0.004) &  1.45 (0.052) &  1.85 (0.061) &  9.36 (2.524) &           1.29 (0.008) &   9.89 (1.547) &  41.07 (7.291) &  32912.25 (34674.225) \\
        & 700 &       1.11 (0.007) &  1.44 (0.048) &  1.85 (0.054) &   9.6 (2.445) &            1.3 (0.005) &  10.11 (0.874) &  42.67 (6.173) &  31758.87 (40191.444) \\
$\Sph^9$ & 200 &       0.31 (0.018) &    1.4 (0.14) &  1.98 (0.199) &  3.28 (0.219) &           0.35 (0.014) &   1.66 (0.113) &   2.37 (0.148) &          4.07 (0.334) \\
        & 500 &       0.31 (0.021) &  1.36 (0.067) &  1.95 (0.124) &  3.25 (0.226) &           0.35 (0.015) &   1.64 (0.083) &   2.34 (0.081) &          4.08 (0.204) \\
        & 700 &        0.3 (0.007) &  1.37 (0.065) &  1.91 (0.074) &  3.29 (0.093) &           0.35 (0.013) &    1.64 (0.09) &   2.32 (0.047) &          4.11 (0.139) \\
$\St{5}{3}$ & 200 &       0.44 (0.023) &  0.38 (0.017) &   0.38 (0.03) &  0.37 (0.009) &           0.52 (0.028) &    0.89 (0.03) &    1.34 (0.05) &          9.91 (0.479) \\
        & 500 &        0.45 (0.01) &  0.37 (0.013) &  0.37 (0.025) &  0.37 (0.014) &           0.51 (0.017) &   0.88 (0.034) &   1.37 (0.038) &            9.94 (0.3) \\
        & 700 &       0.45 (0.029) &  0.37 (0.028) &  0.38 (0.012) &  0.38 (0.017) &           0.52 (0.026) &    0.88 (0.04) &   1.35 (0.032) &          9.75 (0.156) \\
\bottomrule
\end{tabular}}
  \caption{Mean expected values of Riemannian Brownian simulations using projection and random walk approximations. We simulate 1000 paths, for each cost type (first row) and final time (second row) using a a variety of methods. The time increment $h$ is the final time is divided by \# div intervals in the stochastic integration. The mean is over all simulation methods, the numbers in brackets are the variations in the difference between the highest estimates and lowest estimates, (also over all simulation methods).}
\label{tbl:tblBySimType}  
\end{table}
\subsection{Convergence to uniform distributions on compact Riemannian manifolds}
Consider the volume form $d\cM = \det(\sfg(x))^{\frac{1}{2}}dx^{\dim\cM}$ of a Riemannian manifold $\cM$, with the volume $V(\cM)=\int_{\cM}d \cM$. When this is finite, for example, or a compact manifold, the normalized measure $\frac{1}{V(\cM)}d \cM$ is the uniform measure on $\cM$. It is known \cite{Saloff1994}, for the compact case if $X_t$ follows a Riemannian Brownian motion, the distribution of $X_T$ for large $T$ converges to the uniform distribution. For manifolds defined by orthogonal constraints, specifically, for $\SO(N), \St{n}{p}, \Gr{n}{p}$ and the sphere identified with $\St{n}{1}$, there is an effective sampling method for the uniform distribution if the spaces are equipped with there standard metrics ($\SO(N)$ with the trace metric, $\St{n}{p}$ with the canonical metric $\alpha_0=1,\alpha_1 = \frac{1}{2}$ and $\Gr{n}{p}$ with the quotient of that metric.) Specifically, if we equip the ambient space $\cE=\R^{n\times p}$ ($p=n=N$ for $\SO(N)$) with the normal distribution $N(0, \dI_{\cE})$ and $A\in\cE$ is a matrix of size $n\times p$ then $A(A^{\sfT}A)^{-1}A^{\sfT}\in\cM$ follows the Riemannian uniform distribution. This allows a direct sampling of the uniform distribution on these spaces.

For a compact Lie group $\rG$, a left-invariant metric has the volume form proportional to $\det(\cI)$ times the volume form of the bi-invariant (trace) metric. Thus, when normalized, it becomes the left-invariant Haar measure on $\rG$, and for a compact Lie group, the left and right Haar measures are identical. Thus, the large $T$ limit for {\it any} left-invariant $\SO(N)$ metric gives the same uniform distribution under the bi-invariant metric. For a homogeneous space such as $\St{n}{p}$, the family of metrics for different $(\alpha_0, \alpha_1)$ is a quotient of a left-invariant metric on $\SO(n)$, thus, they also give the same large $T$ limit measure. Further, a change of variable shows the expectations of $f$ and $f\circ (X\mapsto X^{-1})$ are the same in the $\SO(N)$ case.

We verify these observations. In \cref{tab:uniform}, we consider the expectation $\E (f(X_T))$ under the Riemannian Brownian measure at $T=40$, starting at $X_0 = \dI_N$ for $\SO(N)$ and $\dI_{n,p} =\begin{bmatrix}\dI_p\\ 0_{(n-p)\times p} \end{bmatrix}$ otherwise\hfill\break
1. $\SO(3)$ and $SO(4)$ equipped with randomly generated left-invariant metrics, and four costs $f_1(X)=|X_{11}|^2, f_2(X)=\sum_{ij=1}^{N} |X_{ij}|, f_3(X)= e^{\frac{1}{2}\sum |X_{ij}|}$, $f_4(X)=(1+\sum |X_{ij}|)^{-\frac{1}{2}}$,
 together with $f_j\circ (X\mapsto X^{-1})$ for $j=1,\cdots, 4$,\hfill\break
 2. $\St{5}{3}$ with $(\alpha_0, \alpha_1) \in \{(1, \frac{1}{2}), (1, \frac{4}{5}), (1, 1)\}$ and $\St{5}{1}$ (the unit sphere) with $f_1(X)=|X_{11}|^2$, $f_2(X)=\sum \lvert X_{ij}\rvert$, $f_3(X)=e^{\frac{1}{2}\sum\lvert X_{ij}\rvert}$, $f_4(X)=(1+ \sum\lvert X_{ij}\rvert)^{-\frac{1}{2}}$,\hfill\break
3. $\Gr{5}{3}$ with $f_1(X) =(XX^{\sfT})_{11}^2$, $f_2(X)=\sum_{ij}\lvert (XX^{\sfT})_{ij}\rvert$, $f_3(X)=e^{\frac{1}{2}\sum\lvert (XX^{\sfT})_{ij}\rvert}$, $f_4(X)=(1+\sum\lvert (XX^{\sfT})_{ij}\rvert)^{-\frac{1}{2}}$,\hfill\break
and compare with the expectations using direct sampling in \cref{tab:uniform}
\begin{table}
  \resizebox{\textwidth}{!}{  
  \begin{tabular}{llllllllll}
\toprule
                       f & \multicolumn{2}{l}{1} & \multicolumn{2}{l}{2} & \multicolumn{2}{l}{3} & \multicolumn{2}{l}{4} \\
                Manifold &         Sample &                          Brown sim. &         Sample &                          Brown sim. &           Sample &                            Brown sim. &         Sample &                          Brown sim. \\
\midrule
             $\Gr{5}{3}$ &          0.429 &                0.429(0.0146) &          6.389 &                6.406(0.0291) &           25.345 &                 25.378(0.1717) &          0.369 &                0.369(0.0008) \\
$\SO(3)$ &   0.332 &   0.329(0.002) &   4.500 &   4.506(0.010) &     9.618 &     9.603(0.013)  &   0.427 &   0.427(0.001) \\
\; Use $f(x^{-1})$&  0.332 &  0.326(0.013) &  4.500 &  4.497(0.021) &  9.618 &  9.591(0.023) &  0.427 &  0.427(0.001) \\
$\SO(4)$ &   0.249 &   0.253(0.010) &   6.791 &   6.806(0.005) &   30.423 &   30.527(0.217) &   0.359 &   0.358(0.000) \\
\; Use $f(x^{-1})$ &  0.249 &  0.252(0.017) &  6.791 &  6.780(0.002) &  30.423 &  30.259(0.612) &  0.359 &  0.358(0.000)\\
                   $\Sph^5$ &          0.199 &                0.198(0.0053) &          1.875 &                1.872(0.0044) &            2.563 &                  2.570(0.0235) &          0.591 &                0.591(0.0006) \\
  $\St{5}{3}\alpha=(1, \frac{1}{2})$  &          0.200 &                0.203(0.0074) &          5.625 &                5.610(0.0035) &           16.876 &                 16.931(0.1065) &          0.389 &                0.389(0.0001) \\
  $\St{5}{3}\alpha=(1, \frac{4}{5})$ &          0.200 &                0.201(0.0015) &          5.625 &                5.641(0.0136) &           16.876 &                 16.965(0.0579) &          0.389 &                0.388(0.0006) \\
  $\St{5}{3}\alpha=(1, 1)$ &          0.200 &                0.193(0.0029) &          5.625 &                5.623(0.0042) &           16.876 &                 16.871(0.0453) &          0.389 &                0.389(0.0003) \\
\bottomrule
  \end{tabular}}
  \caption{Expected values of $f_i, i=1\cdots 4$ as described, comparing direct sampling of uniform distributions on manifolds versus Riemannian Brownian motion simulation at $T=40$. For Brownian simulations, the numbers in the bracket describe the range (max - min) among the simulation methods used (geodesic, It{\=o} and Stratonovich). For $\SO(3)$ and $\SO(4)$, the second lines show sampled\slash Brownian simulated values of $f_i(x^{-1})$.}
  \label{tab:uniform}  
\end{table}  

The results suggest using the large $T$ limit as a sampling method for the uniform distribution on a compact Riemannian manifold could be effective if $\Pi, \Gamma$, and a retraction are available. 
\section{Conclusion}In this paper, we provide a framework to study stochastic differential equations in the embedded framework, with a focus on Riemannian Brownian motion. The methods allow us to simulate Brownian motions on several important manifolds in applied mathematics. We expect the methods to be useful for the research community. In future research, we look forward to applying the methods described here to engineering and scientific problems.
\section{Acknowledgements}Du Nguyen would like to thank Professor Victor Solo for sharing his works on SDE on manifolds. He thanks his family for their loving support and his friend John Tillinghast for helpful comments. The second author is supported by a research grant (VIL40582) from VILLUM FONDEN and the Novo Nordisk Foundation grant NNF18OC0052000.
\bibliographystyle{elsarticle-num}
\bibliography{retraction}

\end{document}